\documentclass[a4paper,10pt]{article}
\usepackage[margin=1in]{geometry}
\usepackage{amsmath, amsthm, amsfonts, stmaryrd}
\usepackage[utf8]{inputenc}
\usepackage{mathrsfs}
\usepackage[active]{srcltx}
\usepackage{color}

\newcommand{\dd}{\mathrm{d}}
\newcommand{\ffi}{\varphi}
\newcommand{\ep}{\varepsilon}
\newcommand{\vv}{\boldsymbol{v}}
\newcommand{\ww}{\boldsymbol{w}}
\newcommand{\rr}{\boldsymbol{r}}
\newcommand{\uu}{\boldsymbol{u}}
\newcommand{\cchi}{\boldsymbol{\chi}}

\newcommand{\Rr}{\boldsymbol{R}}
\newcommand{\Qq}{\boldsymbol{Q}}

\def\N{\mathbb{N}}
\def\R{\mathbb{R}}
\def\L{\textnormal{L}}
\def\P{\textnormal{P}}

\def\W{\textnormal{W}}

\def\E{\textnormal{E}}
\def\H{\textnormal{H}}
\def\V{\textnormal{V}}
\def\pa{\partial}
\def\var{\varepsilon}
\def\Id{\textnormal{Id}}

\newtheorem{Theo}{Theorem}[section]

\newtheorem{lem}[Theo]{Lemma}
\newtheorem{Propo}[Theo]{Proposition}
\newtheorem{definition}[Theo]{Definition}

\begin{document}


\begin{center}
\Large{Entropy, Duality and Cross Diffusion}
 \end{center}
\bigskip


\medskip
\centerline{\scshape L. Desvillettes}
\medskip
{\footnotesize
  \centerline{CMLA, ENS Cachan \& CNRS}
  \centerline{61 Av. du Pdt. Wilson, 94235 Cachan Cedex, France}
\centerline{E-mail : desville@cmla.ens-cachan.fr}}
\bigskip

\centerline{\scshape Th. Lepoutre}
\medskip
{\footnotesize
\centerline{INRIA Rh\^one Alpes}
\centerline{Project-team DRACULA}
\centerline{Universit\'e de Lyon}
\centerline{CNRS UMR 5208}
\centerline{Universit\'e Lyon 1}
\centerline{Institut Camille Jordan}
\centerline{43 blvd. du 11 novembre 1918}
\centerline{F-69622 Villeurbanne cedex France}
\centerline{E-mail : thomas.lepoutre@inria.fr}}
\bigskip

\centerline{\scshape A. Moussa}
\medskip
{\footnotesize
\centerline{UPMC Université Paris 06 \& CNRS}
\centerline{UMR 7598, LJLL, F-75005, Paris, France}
\centerline{\& INRIA Paris-Rocquencourt, Équipe-Projet REO}
\centerline{BP 105, F-78153 Le Chesnay Cedex, France}
\centerline{E-mail : moussa@ann.jussieu.fr}
\bigskip
\centerline{\today}
\medskip






\begin{abstract}
This paper is devoted to the use of the entropy and duality methods for the existence theory of 
reaction-cross diffusion systems consisting of two equations, in any dimension of space. Those systems
appear in population dynamics when the diffusion rates of individuals of two species depend on the concentration of individuals of the same species (self-diffusion), or of the other species (cross diffusion). 
\end{abstract}

\section{Introduction}\label{sec:system}

We are interested in populations consisting of individuals belonging to two distinct species (or two classes in the same species) of (typically) animals, and interacting through diffusion and competition.
\par
The unknowns are the concentration (number density) of individuals of the first species $u_1:=u_1(t,x)\ge0$ and of the second species $u_2 := u_2(t,x)\ge 0$. 
\par

In absence of competition, the
respective populations $u_1$ and $u_2$ would grow at a (strictly positive) rate $r_1$ and $r_2$.
\par 
The competition is taken into account through logistic-type terms, in such a way that the growth rate
becomes $r_1-s_{11}(u_1)-s_{12}(u_2)$ for the first species, and $r_2-s_{22}(u_2)-s_{21}(u_1)$ for
the second one.  The (nonnegative) terms $s_{11}(u_1)$ and $s_{22}(u_2)$ are called intraspecific competition, while
the (nonnegative) terms $s_{12}(u_2)$ and $s_{21}(u_1)$ are by definition the interspecific competition.
Since we are interested here in situations when there is no cooperation (or predator-prey type interaction), we shall assume in the sequel that all functions $s_{ij}$ are nonnegative and increasing (see \cite{Murray}, p.94).
\par
We also assume that the individuals of the two species diffuse with a rate that depends on the
number density of both species: we denote $d_{11}(u_1) + a_{12}(u_2)$ the diffusion rate of the first
species, and $d_{22}(u_2) + a_{21}(u_1)$ the diffusion rate of the second species. The diffusion terms
related to $d_{11}(u_1)$ and $d_{22}(u_2)$ are called self diffusion, those related to $a_{12}(u_2)$ and
$a_{21}(u_1)$ are called cross diffusion. We are interested in this paper in the case when all the functions $d_{ii}$, $a_{ij}$ are nondecreasing (that is, all individuals try to leave the points where the
competition is highest). This type of model were introduced in \cite{Shigesada1979}. 
\par
In the sequel, we systematically denote 
\begin{equation}
\label{eq:aii_dii}
 a_{ii}(u_i) = u_i\, d_{ii}(u_i).
\end{equation}
 We are led to write down the following system (on $\R_+\times\Omega$, 
where $\Omega$ is a smooth bounded open subset of $\R^N$ with outward unit normal vector $n(x)$ at point
$x\in\pa\Omega$):
\begin{align}\label{eq:system:total1}
 \partial_t u_1 - \Delta\big[a_{11}(u_1)+u_1a_{12}(u_2)\big]
 &= u_1\big(r_1-s_{11}(u_1)-s_{12}(u_2)\big):=R_{12}(u_1,u_2), \\
\label{eq:system:total2} \partial_t u_2 - \Delta\big[a_{22}(u_2)+u_2a_{21}(u_1)\big]
 &= u_2\big(r_2-s_{22}(u_2)-s_{21}(u_1)\big):=R_{21}(u_2,u_1),
\end{align}
with the homogeneous Neumann boundary conditions 
\begin{equation}\label{HN}
\nabla u_1 \cdot n(x) = \nabla u_2 \cdot n(x) = 0
 \quad {\hbox{ on }} \quad  \R_+\times\partial\Omega,
\end{equation}
and the initial condition 
\begin{equation}\label{CI}
 u_1 (0, \cdot) = u_1^{in}\ge 0,  \qquad u_2 (0, \cdot) = u_2^{in}\ge 0.
\end{equation}
  The case treated in this paper corresponds to a situation where the reaction terms
 are strictly subquadratic or dominated by the self diffusion, and where the cross diffusions are
subquadratic. Extensions to other cases with faster growth of reaction terms or cross diffusion pressure will be studied in a future work. 
\bigskip

 We detail below the set of mathematical assumptions that will be imposed on the coefficients 
 $a_{ij}$, $s_{ij}$, and which correspond to the case described above:

\begin{itemize}
\item[\textbf{H1}] For $i,j \in\{1,2\}$, $r_i\in \R_+$ and $s_{ij}$ is a nonnegative continuous  function on $\R_+$ which is either strictly sublinear: $\lim_{x \to +\infty} \frac{s_{ij}(x)}{x} = 0$; or
dominated by the self diffusion in the following sense: $\lim_{x \to +\infty} \frac{s_{ii}(x)}{x\, d_{ii}(x)} = 0$ and (for $i\not=j$) $\lim_{x \to +\infty} \frac{s_{ij}(x)}{x\, \sqrt{d_{jj}(x)  + a_{ij}(x)} } = 0$; 
\item[\textbf{H2}] For $i\neq j\in\{1,2\}$, $a_{ij}$ is continuous on $\R_+$ and belongs to $\mathscr{C}^2(]0, +\infty[)$. It is also nonnegative, nondecreasing, concave and vanishes at point $0$. Furthermore, there exists $\alpha\in]0,1[$, $C>0$,
such that
\begin{align*}
\forall x\in ]0, +\infty[,\quad x^\alpha a_{ij}'(x) \geq C.
\end{align*}
\item[\textbf{H3}] The self diffusion rate $d_{ii}$ (recall notation (\ref{eq:aii_dii})) is
continuous on $\R_+$ and belongs to $\mathscr{C}^1(]0, +\infty[)$. It is also nonnegative, nondecreasing 
and such that $d_{ii}(0) >0$ (note 
 in particular that $a_{ii}'$ is bounded below by a strictly positive constant). In the proof of our main Theorem, we
 use $d_{ii}\ge 1$ for the sake of readability.
\end{itemize}
 The union of the assumptions \textbf{H1}, \textnormal{\textbf{H2}} and \textnormal{\textbf{H3}} on the parameters [for $i,j \in \{1,2\}$] $r_i$, $s_{ij}$, $a_{ij}$ (and thus $d_{ii}$) will be called in the sequel the \textnormal{\textbf{H}} assumptions.

\medskip

Since we consider the homogeneous Neumann boundary condition, it is useful to introduce the following notation :
\begin{definition}\label{defineu}
 for any space of functions defined on $\Omega$  whose gradient has a well-defined trace on the boundary $\partial\Omega$ (such as $\mathscr{C}^\infty(\overline{\Omega})$ or $\H^2(\Omega)$ for instance), we add the subscript $\nu$ (the former spaces hence become $\mathscr{C}^\infty_\nu(\overline{\Omega})$ or $\H^2_\nu(\Omega)$) to describe the subspace of functions satisfying the homogeneous Neumann boundary condition.
\end{definition}
 \bigskip
 
 Our main Theorem reads:
 
\begin{Theo}
\label{theo:theo}
Let $\Omega$ be a smooth ($\mathscr{C}^2$) bounded open subset of $\R^N$ ($N\ge 1$). Let $\uu^0:=(u_1^{in},u_2^{in})\in \L^2(\Omega)^2$ be a couple of nonnegative functions and assume that assumptions \textnormal{\textbf{H}} are satisfied on the coefficients of the system. 
\par 
Then, for any $T>0$, there exists a couple $\uu:=(u_1,u_2)$ of nonnegative weak solutions
to \eqref{eq:system:total1}-- \eqref{eq:system:total2} -- \eqref{HN} on $[0,T]$
in the following sense: for $i=1,2$ ($j\neq i$),
\begin{equation}\label{eq:estimationL2}
 \int_0^T \int_{\Omega} \bigg[d_{ii}(u_i(t,x)) + a_{ij}(u_j(t,x))\bigg]
\, |u_i(t,x)|^2 \, \dd x\,\dd t < +\infty, 
\end{equation}
and for any $\theta \in {\mathscr{D}}([0, T[; \mathscr{C}_\nu^{\infty}(\overline{\Omega}))$, we have the weak formulation   
\begin{align*}
 & - \int_{\Omega} u_1^{in}(x)\, \theta(0,x)\, \dd x - \int_0^{\infty}\int_{\Omega} u_1(t,x)  \,\partial_t \theta(t,x)\, \dd x\,\dd t \\
 & - \int_0^{\infty}\int_{\Omega} \Delta \theta(t,x)\,  \Big[ a_{11}(u_1(t,x)) + u_1(t,x)\, a_{12}(u_2(t,x))\Big]\, \dd x\,\dd t \\
= &\int_0^{\infty}\int_{\Omega} R_{12}(u_1(t,x),u_2(t,x))\, \dd x\,\dd t,
\end{align*}
and
\begin{align*}
 & - \int_{\Omega} u_2^{in}(x)\, \theta(0,x)\, \dd x - \int_0^{\infty}\int_{\Omega} u_2(t,x) \,\partial_t \theta(t,x)\, \dd x\,\dd t \\
 & - \int_0^{\infty}\int_{\Omega} \Delta \theta(t,x)\,  \Big[ a_{22}(u_2(t,x)) + u_2(t,x)\, a_{21}(u_1(t,x))\Big]\, \dd x\,\dd t \\
= &\int_0^{\infty}\int_{\Omega} R_{21}(u_1(t,x),u_2(t,x))\, \dd x\,\dd t.
\end{align*}
Moreover, denoting $Q_T:=[0,T]\times\Omega$, the following bounds hold:
\begin{align*}
&\left\| \frac{1}{u_1}{a_{21}'}(u_1)\nabla u_1 \right\|_{\L^{2}(Q_T)} + \left\|\frac{1}{u_2}{a_{12}'}(u_2) \nabla u_2 \right\|_{\L^{2}(Q_T)} \leq K_T ( 1+ \|\uu^0\|_{\L^2(\Omega)}),\\
&\left\| \nabla  \sqrt{a_{21}(u_1)} \right\|_{\L^{2}(Q_T)} + \left\|\nabla  \sqrt{a_{12}(u_2)}\right\|_{\L^{2}(Q_T)} + \left\|\nabla  \sqrt{a_{21}(u_1)a_{12}(u_2)}\right\|_{\L^{2}(Q_T)}  \leq K_T ( 1+ \|\uu^0\|_{\L^2(\Omega)}),
\end{align*}
for some positive constant $K_T$ depending only on $T$ and the data of the equations ($a_{ij}$, $r_i$, $s_{ij}$).
\end{Theo}
Note that all the terms in the weak formulation above are well-defined thanks to assumptions \textnormal{\textbf{H}}. Indeed, remembering that $d_{ii}$ is bounded below, estimate \eqref{eq:estimationL2} implies $u_i\in\L^2([0,T]\times \Omega)$ so that the reaction terms are integrable (thanks to \textbf{H1}). Using the identity
$$  s_{12}(u_2)\,u_1 = \left(\frac{s_{12}(u_2)}{u_2\, \sqrt{d_{22}(u_1) + a_{12}(u_2)}} \right)\,
\,  \left( \frac{u_2\,\sqrt{d_{22}(u_2) + a_{12}(u_2)}}{\sqrt{a_{12}(u_2)} } \right)\,\,u_1\,\sqrt{a_{12}(u_2)} , $$
 the terms coming out of cross diffusion are also well defined due to concavity assumptions \textnormal{\textbf{H2}}. Those coming out of self diffusion are integrable thanks to (\ref{eq:estimationL2}).
\bigskip

We now comment the assumptions on the initial data and coefficients of system \eqref{eq:system:total1}--\eqref{eq:system:total2}. The requirement that $\uu^0\in (\L^2(\Omega))^2$ can certainly be relaxed to $\uu^0\in \textnormal{H}^{-1}(\Omega)^2$ and $\psi_i(u_i^{in})\in \L^1(\Omega)$, with $\psi_i$ defined later in Definition~\ref{defips} (this should be enough to keep estimate \eqref{eq:estimationL2}). It is most likely possible to relax 
assumption \textbf{H1}, provided that \textnormal{\textbf{H2}} and \textnormal{\textbf{H3}} are reinforced. 
Finally, there is some hope of treating the special case when reaction terms are exactly quadratic,
 thanks to recent improvements in the theory of duality Lemmas \cite{CDFO}. 
\smallskip
 
We restricted in this paper our study to the case when both $a_{ij}$, $i\neq j$ are concave, whereas our methods should adapt in the case when one of them is concave and the
other one is convex. The corresponding theory is then quite different and will be left to a future work.
Note that when both $a_{ij}$, $i\neq j$ are convex, our feeling is that existence of weak global solutions does
not hold in general.
\smallskip

Our opinion is that for systems involving cross diffusion and consisting of more than two equations, the type of Lyapounov functional that we build in the sequel can exist only for a very small class of cross diffusion terms 
(that is, strong algebraic constraints have to be assumed on the cross diffusion coefficients).
\smallskip

Finally, we do not treat the case when one cross diffusion term is missing (sometimes called the ``triangular case"), since the Lyapounov functional that we introduce in the sequel degenerates in this case.
\bigskip

Let us also explain the meaning of assumptions \textnormal{\textbf{H}} when all the functions appearing in \eqref{eq:system:total1}--\eqref{eq:system:total2} are of the form $x \mapsto x^q$ (or $x \mapsto x + x^q$ for the self diffusion), with $q\in \R$. In this case, if $s_{ij}(x) = S_{ij}\, x^{\sigma_{ij}}$ and $a_{ij}(x) = D_i\, x\, {\delta_{ij}} + A_{ij}\, x^{\alpha_{ij}}$, with $S_{ij}, A_{ij}, D_i>0$ and  $\sigma_{ij}, \alpha_{ij} \in \R$, those assumptions become:
 $$ \forall i=1,2, \qquad  0 \le\sigma_{ii} < \sup(1,\alpha_{ii}); $$
 $$  \forall i \neq j, \qquad 0 \le \sigma_{ij} < \sup\left( \frac{\alpha_{jj} + 1}2, 1 + \frac{\alpha_{ij}}2 \,\right), $$
$$  \qquad 0 < \alpha_{ij} <1 . $$
 \par
  As a consequence, our Theorem provides existence of global weak solutions for systems like 
  \begin{align}\label{eqpa1}
 \partial_t u_1 - \Delta\big[ (D_1 + A_{11}\,u_1^{\alpha} + A_{12}\,u_2^{\beta})\, u_1\big]
 &= u_1\big(r_1-S_{11}\,u_1-S_{12}\,u_2\big), \\
\label{eqpa2} \partial_t u_2 - \Delta\big[ (D_2 + A_{21}\, u_1^{\gamma} + A_{22}\,u_2^{\delta})\, u_2\big]
 &= u_2\big(r_2-S_{21}\,u_1-S_{22}\,u_2\big),
\end{align}
 with $0\le \alpha, \delta<1$; $\beta,\gamma>0$, and $r_i, D_i, A_{ij}, S_{ij}>0$.
 \bigskip
 
 Let us now describe how our work fits in the existing literature.

The question of local and global existence of classical solutions of systems like \eqref{eq:system:total1}--\eqref{eq:system:total2} was treated in many particular cases. Most of them deal either with the case when the system is in fact parabolic (that is, when the diffusion matrix is elliptic), which amounts to assume that self diffusion is dominant w.r.t.
cross diffusion (for instance when $8\,A_{11}\geq A_{12}, 8\,A_{22}\geq A_{21}$ in eq. (\ref{eqpa1}), (\ref{eqpa2}) with
$\alpha=\beta=\gamma=\delta =1$, cf. \cite{Yagi}), or in the triangular case (when $a_{21}(u_1)=0$, cf. \cite{LNW}). The question of local existence is usually treated using Amann's theorem \cite{Amann90b}, and extension to global existence requires additional structure to preserve boundedness of solutions see \cite{Choi2004,Deuring1987} for instance.  
\smallskip

 Our work is concerned with situations that are neither parabolic, nor triangular.
 It extends to a very large class of problems the result of \cite{Chen2006}, dedicated to the 
cross diffusion model for population dynamics introduced by Shigesada, Kawasaki and Teramoto in \cite{Shigesada1979}, where
 the cross diffusions have a linear form, that is $\beta=\gamma=1$ in eq. (\ref{eqpa1}), (\ref{eqpa2}). 
Note that this system can be seen as a limiting case of the equations that we treat.
 \smallskip
 
 Our results rely on two main ingredients: entropy structure and duality Lemmas.
We show that our systems possess a hidden entropy-like structure, strongly reminiscent of the 
entropy structure exhibited in \cite{Chen2006}. 
 In general, this structure however gives less estimates than in \cite{Chen2006}. We therefore need
another ingredient in order to recover existence of weak global solutions, namely duality Lemmas: we
 recall that duality Lemmas enable to recover $\L^2$  type estimate for solutions
 to linear singular parabolic equations (with variable coefficients) when the diffusion rate is inside the Laplacian. This is how estimate \eqref{eq:estimationL2} is derived.
\smallskip

 For the use of an underlying entropy structure, its link with symmetrization, and its applications to existence of weak solutions, we refer to \cite{Kawa_Shiz,Degond_Sym}. The possibility to use such a structure in the case of cross diffusion was  first noticed in \cite{Galiano_Num_Math}, and exploited in \cite{Chen2006,Chen2004,Jungel_Stelzer,Hitmeir}.
  The duality estimate that we use comes from \cite{PiSc}, and was applied together with entropy methods in the framework of reaction-diffusion systems in \cite{Desvillettes2007a}. Is was also applied in the framework of cross diffusion or similar models in \cite{Lepoutre_JMPA, BoPiRo, PierreD}.  
\smallskip

 We finally quote some works dealing with other aspects of cross diffusion models. For modeling issues, we refer
to \cite{Mimura_cross, Desvillettes-Conforto,BoPiRo,Lepoutre_JMPA}.
For the analysis of equilibria, we cite \cite{Mimura_cross,Lou_Martinez} for instance. 
\bigskip

Unfortunately, as often in papers dealing with cross diffusion, the process of approximation enabling to make the estimates and structures rigorous is quite involved, and gives rise to various difficulties which explain the length and technicality of the proofs.  

 Our paper is structured as follows:  in Section \ref{section2}, we introduce notations which are used in the proof of our main Theorem, especially those related to the Lyapounov functional that we systematically use. We also present some classical lemmas used in the sequel. Then, Section \ref{section3} is devoted to the proof of existence of a solution to a finite-dimensional (Galerkin) approximation of a discrete time version of our (smoothed) system. \emph{A priori} estimates (and their dependence with respect to the various approximations) are provided for such solutions. We let the dimension of the Galerkin approximation go
to infinity in Section \ref{section4}. The duality estimate is presented and proven in Section \ref{section5}. The last section (that is, Section \ref{section6}) is devoted to the relaxation of all remaining 
approximations, which leads to the proof of our main Theorem. 
 
\section{Preliminaries}\label{section2}

\subsection{Entropy structure} \label{sub21}

We first introduce some notations which enable to rewrite our system under a form
in which a Lyapounov functional naturally appears.
\bigskip
 
\begin{definition}\label{defips}
For given cross diffusion parameters $a_{12}$ and $a_{21}$ (satisfying assumption
\textnormal{\textbf{H2}}), we introduce $\ffi_i$, $\psi_i$ (for $i=1,2$), as:
 $$ \ffi_1(x):=\int_{1}^x \frac{a'_{21}(t)}{t}\dd t,
 \,\,\ffi_2(x):=\int_{1}^x \frac{a'_{12}(t)}{t}\dd t, $$
$$
 \psi_i(x):=\int_{1}^x \ffi_i(t)\dd t. $$
We also define
$$ w_i(t,x) := \ffi_i(u_i(t,x)).$$

\end{definition}
 
\vspace{2mm}
One can then rewrite the system in a symmetric form:
\begin{align}\label{eq:symetric}
\partial_t \begin{pmatrix}
            u_1 \\
            u_2
           \end{pmatrix}
- \textnormal{div} \stackrel{:=A(u_1,u_2)}{\overbrace{\begin{pmatrix}
\frac{{a_{11}'}(u_1)+{a_{12}}(u_2)}{{a_{21}'}(u_1)}u_1 & u_1 u_2\\
u_1 u_2 & \frac{{a_{22}'}(u_2)+{a_{21}}(u_1)}{{a_{12}'}(u_2)}u_2
\end{pmatrix}}}
\begin{pmatrix}
 \nabla w_1 \\\nabla w_2
\end{pmatrix}
=\begin{pmatrix}
 R_{12}(u_1,u_2) \\
R_{21}(u_2,u_1)
 \end{pmatrix}.
\end{align}
 The terms $\nabla w_i$ have to be considered as scalars for the matrix product, the divergence being understood line by line (after the matrix product).\vspace{2mm}\\
 Multiplying the equations of the system by $w_1,w_2$ and integrating in space we 
 obtain formally the entropy identity
 \begin{equation}
 \label{eq:entropy}
 \frac{d}{dt}\int_\Omega \left(\psi_1(u_1)+\psi_2(u_2) \right) 
+\int_\Omega (\nabla w_1,\nabla w_2)A(u_1,u_2)
\begin{pmatrix}\nabla w_1\\ \nabla w_2\end{pmatrix}=  (w_1, w_2)\begin{pmatrix} R_{12}\\  R_{21}\end{pmatrix}.
 \end{equation}
 
\subsection{Properties of the symmetric matrix} \label{sub22}

For $u_1,u_2>0$, $\uu:=(u_1,u_2)$, the matrix $A(\uu)$ is defined by
\begin{align} \label{mat}
\displaystyle A(u_1,u_2):=\begin{pmatrix}
\frac{{a_{11}'}(u_1)+{a_{12}}(u_2)}{{a_{21}'}(u_1)}u_1 & u_1 u_2\\
u_1 u_2 & \frac{{a_{22}'}(u_2)+{a_{21}}(u_1)}{{a_{12}'}(u_2)}u_2
\end{pmatrix},
\end{align}
and we denote the associate quadratic form $\Qq(\uu)$. For a function $\uu=(u_1,u_2):\Omega\rightarrow ]0, +\infty[^2$, we defined in the previous subsection $\ww=(w_1,w_2)=(\ffi_1(u_1),\ffi_2(u_2))$. We will have to deal with the expression
\begin{align*}
\Qq(\uu)(\nabla \ww),
\end{align*}  
 since this term naturally appears in the entropy estimate. We establish the following proposition, that will be useful later.\vspace{2mm}\\
\begin{Propo}\label{prop:propertyA}
  \noindent Under the \emph{\textbf{H}} assumptions, the application $A:{]0,+\infty[}^2\rightarrow\mathcal{M}_2(\R)$ (defined by (\ref{mat})) belongs to $\mathscr{C}^0({]0,+\infty[}^2)$ and takes its values in $S_2^{++}(\R)$ (space of strictly positive 
symmetric matrices). We have furthermore the two following estimates on the quadratic form $\Qq$
\begin{align}
\label{ineq:propertyA1} \Qq(\uu)(\nabla \ww)  &\geq \frac{1}{u_1}{a_{21}'}(u_1)|\nabla u_1|^2 + \frac{1}{u_2}{a_{12}'}(u_2)|\nabla u_2|^2,\\
\label{ineq:propertyA2}\Qq(\uu)(\nabla\ww)&\geq  4\,\Big\{ \left|\nabla  \sqrt{a_{21}(u_1)}\right|^2 + \left|\nabla  \sqrt{a_{12}(u_2)}\right|^2 + \left|\nabla  \sqrt{a_{21}(u_1)a_{12}(u_2)}\right|^2\Big\}.
\end{align} 
\end{Propo}

\begin{proof}
The fact that $A \in\mathscr{C}^0(]0, +\infty[^2)$ is an easy consequence of the regularity of the functions $a_{ij}$. As for the strict positiveness of the matrix, we just decompose $A$ between self (clearly strictly positive) and cross diffusion:
\begin{align*}
A(u_1,u_2)=\stackrel{B(u_1,u_2)}{\overbrace{\begin{pmatrix}
\frac{{a_{11}'}(u_1)}{{a_{21}'}(u_1)}u_1  & 0  \\
0 & \frac{{a_{22}'}(u_2)}{{a_{12}'}(u_2)}u_2 
\end{pmatrix}}} 
+ \stackrel{C(u_1,u_2)}{\overbrace{\begin{pmatrix}
\frac{{a_{12}}(u_2)}{{a_{21}'}(u_1)}u_1 & u_1u_2  \\
u_1u_2 & \frac{{a_{21}}(u_1)}{{a_{12}'}(u_2)}u_2
\end{pmatrix}}}.
\end{align*}
We see that due to the assumptions on $a_{12}$ and $a_{21}$, the cross diffusion matrix $C(u_1,u_2)$ is nonnegative. Indeed, for $u_1,u_2>0$
\begin{align*}
\frac{{a_{12}}(u_2)}{{a_{21}'}(u_1)}u_1 \times \frac{{a_{21}}(u_1)}{{a_{12}'}(u_2)}u_2 &\geq u_1 u_2 \times u_1 u_2\\
&\Updownarrow\\
a_{12}(u_2)a_{21}(u_1) &\geq u_1 u_2 {a_{12}'}(u_2){a_{21}'}(u_1),
\end{align*}
which is true by concavity. We get
\begin{align*}
\Qq(\uu)(\nabla \ww) = \big[{}^t\!\,\nabla \ww\big]\big[ A(\uu) \nabla \ww\big] \geq \big[{}^t\!\,\nabla \ww \big]\big[B(\uu) \nabla \ww\big] = \frac{{a_{11}'}(u_1)}{{a_{21}'}(u_1)}u_1|\nabla w_1|^2 + \frac{{a_{22}'}(u_2)}{{a_{12}'}(u_2)}u_2|\nabla w_2|^2.
\end{align*}
Since $w_i = \ffi_i(u_i)$, with $\displaystyle \ffi_i'(x) = \frac{a_{ji}'(x)}{x}$, we end up with
\begin{align*}
\Qq(\uu)(\nabla \ww)  \geq \frac{{a_{11}'}(u_1)}{u_1}{a_{21}'}(u_1)|\nabla u_1|^2 + \frac{{a_{22}'}(u_2)}{u_2}{a_{12}'}(u_2)|\nabla u_2|^2,
\end{align*} 
which, using assumption \textnormal{\textbf{H3}}, leads to the first lower bound \eqref{ineq:propertyA1}. On the other hand, expanding directly $\Qq(\uu)(\nabla \ww)$ with the definitions of $A(\uu)$ and $\ww$, we get
\begin{align*}
\Qq(\uu)(\nabla \ww) &= \big[a_{11}'(u_1)+a_{12}(u_2)\big]\frac{a_{21}'(u_1)}{u_1}|\nabla u_1|^2 + \big[a_{22}'(u_2)+a_{21}(u_1)\big]\frac{a_{12}'(u_2)}{u_2}|\nabla u_2|^2 \\
&+ 2\, a_{21}'(u_1)a_{12}'(u_2)\langle \nabla u_1 , \nabla u_2 \rangle.
\end{align*}
The concavity property ${a_{ji}(u_i)\geq a_{ji}'(u_i)u_i}$ used before can be written (here for $i=1$, $j=2$)
\begin{align*}
\frac{a_{21}'(u_1)}{u_1} \geq \frac{{a_{21}'(u_1)}^2}{a_{21}(u_1)},
\end{align*}
and together we the assumption \textnormal{\textbf{H3}}, we get
\begin{align*}
\Qq(\uu)(\nabla \ww) &\geq a_{11}'(u_1)\frac{a_{21}'(u_1)^2}{a_{21}(u_1)}|\nabla u_1|^2 + a_{22}'(u_2)\frac{a_{12}'(u_2)^2}{a_{12}(u_2)}|\nabla u_2|^2 \\
&+ a_{12}(u_2)\frac{a_{21}'(u_1)^2}{a_{21}(u_1)}|\nabla u_1|^2 + a_{21}(u_1)\frac{a_{12}'(u_2)^2}{a_{12}(u_2)}|\nabla u_2|^2 + 2 a_{21}'(u_1)a_{12}'(u_2)\langle \nabla u_1 , \nabla u_2 \rangle\\
&=4  |\nabla\sqrt{a_{21}(u_1)}|^2 + 4  |\nabla\sqrt{a_{12}(u_2)}|^2 \\
&+ 4|\sqrt{a_{12}(u_2)}\nabla\sqrt{a_{21}(u_1)}|^2 +  4|\sqrt{a_{21}(u_1)}\nabla\sqrt{a_{12}(u_2)}|^2 \\
&+ 4\times 2 \sqrt{a_{21}(u_1)} \sqrt{a_{12}(u_2)} \langle \nabla \sqrt{a_{21}(u_1)}, \nabla \sqrt{a_{12}(u_2)}\rangle,
\end{align*}
and we thus end up with eq. \eqref{ineq:propertyA2}.
\end{proof}


\subsection{Properties of the functions $\ffi_i$ and $\psi_i$} \label{sub23}

We shall need in the sequel the following elementary result:

\begin{lem}\label{lem:convcav}
 Take $h,\ell\in\mathscr{C}^0(\R_+)\cap\mathscr{C}^1(]0, +\infty[)$ with $h$ concave and $\ell$ nonnegative and convex,
 with $\ell'(x)>0$ for all $x$ large enough.
 \par
 Then there exists a constant $A_{h,\ell}>0$ such that $h(x)\leq A_{h,\ell}\,(1+\ell(x))$ for all $x\in\R_+$.
\end{lem}
\begin{proof}
If $h$ is bounded from above, then $A_{h,l}=\sup h$ works. Otherwise, $h'>0$ on $\R_+$.  
 Say that $0<\ell'(x)$ for $x>M$. Then for all $x\in\R_+$, $x-M\leq \frac{\ell(x)-\ell(M)}{\ell'(M)}$. Then
 since $h'>0$ and $h$ is concave, we can write 
 $$
 h(x)\leq h(M)+h'(M)(x-M)\leq h(M)+\frac{h'(M)}{\ell'(M)}(\ell(x)-\ell(M))\leq h(M)+\frac{h'(M)}{\ell'(M)}\ell(x),
 $$
 so that the constant $A_{h,l}=\max\left(h(M),\frac{h'(M)}{\ell'(M)}\right)$ works.
 That concludes the Proof of Lemma \ref{lem:convcav}.
\end{proof}

Using the previous assumptions on the $a_{ij}$, one can see that the $\ffi_i,\psi_i$ at least belong to$\mathscr{C}^2(]0,+\infty[)$.
The $\psi_i$ are convex functions and the $\ffi_i$ strictly nondecreasing. Moreover:
 
\begin{lem}\label{lem:ineqconv}
We assume \textnormal{\textbf{H2}}  on the coefficients $a_{ij}$ ($i \neq j$). 
Then the $\psi_i$ are convex functions and the $\ffi_i$ are strictly nondecreasing.
 Moreover (for $i\neq j$):
\begin{itemize}
 \item[(i)] For all $x,y\in\R_+$, $\psi'_i(x)(x-y)\geq \psi_i(x)-\psi_i(y)$.
\item[(ii)] $\psi'_i(x)=\operatorname*{o}_{x\rightarrow 0^+}(1/x)$ and hence $x\psi_i'(x)\geq B$, for some constant $B<0$  for all $x\in ]0, +\infty[$.
\item[(iii)] $\psi_i$ has a limit at point $0^+$ ($\psi_1(0)=a_{21}(1)$),
 furthermore $\psi_i$ is strictly positive on $\R_+$.
\item[(iv)] There exists a constant $D>0$ such that, for all $x\in\R_+$
\begin{align*}
  \forall \alpha\in[0,1], \quad x^\alpha + a_{ji}(x) &\leq D\,(1+\psi_i(x)),\\
\quad x\psi_i'(x) &\leq D\,(1+\psi_i(x)). 
\end{align*} 
\end{itemize}
\end{lem}
\begin{proof}
Let us treat only the case $i=1$, the other one being similar.
\begin{itemize}
 \item[(i)] $\psi_1$ is convex.
\item[(ii)]  $\psi_1'(x)=\ffi_1(x)$ and
\begin{align*}
 \ffi_1(x)=\int_1^x \frac{a_{21}'(t)}{t}\dd t = \Big[\frac{a_{21}(t)}{t}\Big]_1^x + \int_{1}^x \frac{a_{21}(t)}{t^2}\dd t = \operatorname*{o}_{x\rightarrow 0^+}(1/x),
\end{align*}
since $a_{21}\in\mathscr{C}^0(\R_+)$ and $a_{21}(0)=0$. 
The function  $x \mapsto x\ffi_1(x)$ is strictly nondecreasing after $x=1$ and bounded near $0$, hence lower bounded.
\item[(iii)] Note that $\psi_1''(t)t=a_{21}'(t)$. We have
\begin{align*}
          a_{21}(1) - a_{21}(x) = \int_x^1 t\psi_1''(t)\dd t = \Big[t\psi_1'(t)\Big]_x^1 - \int_x^1 \psi_1'(t)\dd t,\\
             \end{align*}
             then noticing that $\psi_i(1)=\varphi_i(1)=0$ by construction, we have
             $$
             a_{21}(1)-a_{21}(x)=\psi_1(x)-x\varphi_1(x).
             $$
hence the previous point $(ii)$ gives the limit near $0$. For the positiveness, just notice that $\ffi_i = \psi_i'$ is negative on $[0,1]$.
\item[(iv)] We use Lemma \ref{lem:convcav} with $h(x)=x + a_{ji}(x)$, $\ell(x)=\psi_1(x)$ and $x^\alpha \leq 1+x$ 
in order to obtain the first inequality. For the second inequality, we use the same Lemma with $h(x)=x\psi_1'(x)-2\psi_1(x)$, and $\ell(x)=\psi_1(x)$.
\end{itemize}

\end{proof}

\subsection{A small perturbation} \label{sub24}

Since $\ffi_i$ may not be one to one, we use a small perturbation of this function and
consider the following definition:

\begin{definition}\label{defipseps}
Let us assume \textnormal{\textbf{H2}} (and recall Definition \ref{defips}) on the coefficients $a_{ij}$ ($i \neq j$), and
\textnormal{\textbf{H3}} on the coefficients $a_{ii}$.
\par
For all $\var>0$ (small enough), we introduce
\begin{align*}
 \ffi_i^\ep(x):=\ffi_i(x)+\ep\ln(x),
\end{align*}
and, equivalently, (for $i\neq j$) $a_{ij}^\ep:=a_{ij}+\ep x$, and $\psi_i^\ep(x):=\psi_i(x)+\ep x\ln(x)-\ep x$.
\par
We also introduce $a_{ii}^\ep(x) = x \,d_{ii}^\ep(x)$, with $d_{ii}^\ep=\gamma_\ep(d_{ii})$,
 where $(\gamma_\ep)_{\ep>0}$ is an increasing family of smooth nonnegative and nondecreasing functions,
 such that $\gamma_\ep \leq \ep^{-1}$ on $\R_+$, $\gamma_\ep(1)=1$ and  $(\gamma_\ep)_{\ep>0}$
is uniformly converging to the identity on compact sets.
\par
Finally, we denote by $A^\ep$ the matrix $A$ (defined by (\ref{mat})), where the coefficients $a_{ij}$ are replaced
by $a_{ij}^\ep$ (for all $i,j \in \{1,2\}$).
\end{definition}

\bigskip

The crucial point is the following: Proposition \ref{prop:propertyA} is still true when one replaces the functions
 $a_{ij}$ and $A$ by their $\ep$-approximations,  $a_{ij}^\ep$ and $A^\ep$, (assumption \textnormal{\textbf{H2}} and \textnormal{\textbf{H3}} hold for the $a_{ij}^\ep$), and if one tries to reproduce Lemma \ref{lem:ineqconv} with these new functions, all the inequalities remain the same, the constants being a little bit changed but \sl{ not depending on $\ep$}.
We write the following lemma, which summarizes the situation:

\begin{lem}\label{lem:ineqconv:approx} We assume that \textnormal{\textbf{H2}} and 
\textnormal{\textbf{H3}} holds on the coefficients $a_{ij}$ (for all $i,j \in \{1,2\}$), and use the notations of 
Definition \ref{defipseps}. 
\par
Then, the coefficients $a_{ij}^\ep$ also satisfy \textnormal{\textbf{H2}} and 
\textnormal{\textbf{H3}} (with constants independent of $\var$ for $\var>0$ small enough).
Also 
$B$ and $D$ (the constants of Lemma \ref{lem:ineqconv}) may be changed in order to have, for $i=1,2$ and $0<\ep<1$:
\begin{itemize}
 \item[(i)] For all $x,y\in\R_+$, ${\psi_i^{\ep}}'(x)(x-y)\geq \psi_i^\ep(x)-\psi_i^\ep(y)$.
\item[(ii)] ${\psi^\ep_i}'(x)=\operatorname*{o}_{x\rightarrow 0^+}(1/x)$ and furthermore,
 for all $x\in ]0, +\infty[$, $x{\psi_i^{\ep}}'(x)\geq B-\ep e^{-1}$.
\item[(iii)] $\psi_i^\ep$ has a limit at point $0^+$. For $\ep$ small enough, $\psi_i^\ep$ is strictly
positive on $\R_+$.
\item[(iv)] For all $x\in\R_+$
\begin{align*}
  \forall \alpha\in[0,1], \quad x^\alpha + a_{ji}^\ep(x) &\leq D(1+\epsilon)\,(1+\psi_i^\var(x)), \\
 \quad x{\psi_i^\ep}'(x) &\leq D(1+\epsilon)(1+\psi_i^\ep(x)), 
\end{align*} 
where $D$ is the constant defined above in lemma~\ref{lem:ineqconv}.
\end{itemize}
\end{lem}
\begin{proof}
As before, we treat only the case $i=1$.
 \begin{itemize}
 \item[(i)] $\psi_1^\ep$ is still convex.
\item[(ii)]  $\ln(x)=\operatorname*{o}_{x\rightarrow 0^+}(1/x)$ and $\ep x\ln(x)\geq -\ep e^{-1} \geq - e^{-1}$.
\item[(iii)] $x\ln(x)-x$ goes to $0$ with $x$ and $\ep(x\ln(x)-x)\geq -\ep$.
\item[(iv)] Firstly, we notice that $D\geq 1+a_{ji}(1)>1$. For the first inequality,  is sufficient to notice that since $\psi_i^\varepsilon\leq \psi_i$, we have
\begin{align*}
\ep x &\leq \ep (x+a_{ij}(x))\leq \ep D(1+\psi_i),
\end{align*}
and we conclude with 
$$
x^\alpha+a_{ij}^\ep(x)\leq D(1+\ep)(1+\psi_i(x))\leq D(1+\ep)(1+\psi_i^\ep(x).
$$
For the second inequality, we use
\begin{align*}
\ep x \ln(x) &\leq \ep (x\ln(x)-x+1)+\ep x\leq D\ep (x\ln(x)-x+1)+D(1+\psi_i(x)),
\end{align*} 
so that 
$$ x{\psi_i^\ep}'(x)\leq x\psi_i'(x)+\ep x \ln(x) \leq D(1+\psi_i(x))+ D\ep (x\ln(x)-x+1)+D(1+\psi_i(x))\leq D(1+\ep)(1+\psi_i^\ep(x))
$$
for some constant $D$ (obviously not depending on $\ep$). Then 
we add these inequalities to those of point (iv) of Lemma \ref{lem:ineqconv}, that we already proved.
\end{itemize}
\end{proof}

\subsection{Two standard results} \label{sub25}

 The preliminaries are concluded by the statement of two results taken from the existing literature, and that will be used in the sequel.

\begin{lem}[Discrete Gr\"onwall]\label{lem:grondis}
Consider two nonnegative sequences $(v_n,w_n)_{n\in\N}$, satisfying for some positive constants $C>0,\theta \in ]0,1[$, 
\begin{align*}
 \forall n\in\N^*,\quad v_{n}\leq v_{n-1} + \theta v_n + w_{n}.
\end{align*}
Then, for all $n\in\N^*$:
\begin{align*}
 v_n \leq e^{n\lambda_\theta}v_0 + \sum_{k=0}^{n-1} e^{k\lambda_\theta}w_{n-k}\leq e^{\lambda_\theta}\Big[v_0+\sum_{k=1}^{n}e^{(-k+1)\lambda_\theta }w_k\Big],
\end{align*}
with $\lambda_\theta:=\theta/(1-\theta)$. If $w_n=C$ is constant, we have
$$
v_n\leq e^{n\lambda_\theta}\left[v_0 + \frac{C}{\theta}\right].
$$
\end{lem}
\begin{proof}
Notice first that $\displaystyle \frac{1}{1-\theta} \leq e^{\lambda_\theta}$ by the convexity inequality of the exponential.  We hence have 
\begin{align*}
v_n \leq e^{\lambda_\theta} v_{n-1} + \frac{w_n}{1-\theta}.
\end{align*}
By straightforward induction, we get
\begin{align*}
v_n \leq \frac{1}{(1-\theta)^n} v_0 +  \sum_{k=1}^{n} \frac{w_{k}}{(1-\theta)^{n-k+1}}  \leq  \frac{1}{(1-\theta)^n} \left[v_0 +\sum_{k=1}^n\frac{1}{(1-\theta)^{1-k}}w_k\right]\leq e^{n\lambda_\theta} \left[v_0 +\sum_{k=1}^n w_{k}e^{(1-k)\lambda_\theta}\right].
\end{align*}
In case of constant $w_n$, we use the following fact
$$
\frac{1}{(1-\theta)^n} \left[v_0 +\sum_{k=1}^n\frac{1}{(1-\theta)^{1-k}}C\right]\leq e^{n\lambda_\theta}\leq \frac{1}{(1-\theta)^n} \left[v_0 +\frac{C}{\theta}\right]\leq e^{n\lambda_\theta}\left[v_0+\frac{C}{\theta}\right] .
$$
\end{proof}
The following Theorem can be found in \cite{gilbarg} (in the more general case of an infinite dimensional Banach space) where it is presented as the ``Leray-Schauder Theorem'' (p.286, Theorem 11.6):

\begin{Theo}
\label{theo:brou}
 Let $(\E,\|\cdot\|)$ be a Euclidian vector space and \\
$T:[0,1]\times \E\rightarrow\E$ a continuous function satisfying $T(0,\cdot)\equiv 0$. Suppose furthermore the existence of $R>0$ such that for any $s\in[0,1]$, the following \emph{a priori} estimate holds
 for the fixed points of $T(s,\cdot)$:
\begin{align*}
 T(s,x)=x \Longrightarrow \|x\| < R.
\end{align*}
Then $T(1,\cdot):\E\rightarrow\E$ has at least one fixed point in $B(0,R)$.
\end{Theo}

\section{Approximate system of finite dimension}\label{section3}

\subsection{Notations}\label{sub31}

We start with a definition related to the discretization w.r.t. time, as in \cite{Chen2006}:

\begin{definition} \label{timestep}
We decompose the time interval, $\displaystyle{]0,T]=\bigcup_{k=1}^N](k-1)\tau,k\tau]}$, where $N\in\N^*$ and $\tau:=T/N$,
and introduce the finite difference operator : $\displaystyle{{\partial_\tau u^{k}:=\frac{u^{k}-u^{k-1}}{\tau}} }$.
 \end{definition}

We also introduce new definitions related to the reaction terms:

\begin{definition} \label{reactionpm}
 Let us denote by $R_{12}^\pm$ and $R_{21}^\pm$ the positive/negative parts of the source terms, and use the same notation for $\Rr^\pm$. To write things more precisely, we have
\begin{align*}
\Rr^+(\uu):=\begin{pmatrix}
r_1 & 0 \\
0 & r_2
 \end{pmatrix}\begin{pmatrix}
        u_1\\
u_2
 \end{pmatrix},\quad \Rr^-(\uu):=
\begin{pmatrix}
u_1 & 0 \\
0 & u_2
 \end{pmatrix}
\begin{pmatrix}
        s_{11}(u_1)+s_{12}(u_2)\\
        s_{22}(u_2)+s_{21}(u_1)
 \end{pmatrix},
\end{align*}
and obviously
\begin{align*}
\Rr(\uu) = \Rr^+(\uu)-\Rr^-(\uu).
\end{align*}
We introduce the following approximation for this reaction term
\begin{align*}
\Rr^\ep(\uu) = \Rr^+(\uu)-\Rr^{-,\ep}(\uu),
\end{align*}
with
\begin{align*}
\Rr^{-,\ep}(\uu):=
\begin{pmatrix}
u_1 & 0 \\
0 & u_2
 \end{pmatrix}
\begin{pmatrix}
        \gamma_\ep\big(s_{11}(u_1)+s_{12}(u_2)\big)\\
\gamma_\ep\big(s_{22}(u_2)+s_{21}(u_1) \big)
 \end{pmatrix},
\end{align*}
where $\gamma_\ep$ is the truncation function that we used in Definition \ref{defipseps}.
\end{definition}
Let us consider a sequence $(\V_n)_{n\in\N}$ of subspaces of $\mathscr{C}^\infty_{\nu}(\overline{\Omega})$, such that for all $n$, $\V_n$ is $n$-dimensional, $\V_n\subset\V_{n+1}$, and $\bigcup_{n\in\N} \V_n$ is  dense in $\L^2(\Omega)$. We 
assume that $\V_1 = \langle \mathbf{1}_\Omega\rangle_\R$, so that the constant function $\mathbf{1}_{\Omega}$ lies in all subspaces $\V_n$. 
\vspace{2mm}\\
For the sake of clarity, let us denote in boldface the two component vectors:
\begin{align*}
\cchi:=\begin{pmatrix}
        \chi_1 \\ \chi_2
       \end{pmatrix},
\uu:=\begin{pmatrix}
        u_1 \\ u_2
       \end{pmatrix}
,
\ww:=\begin{pmatrix}
        w_1 \\ w_2
       \end{pmatrix}\quad\text{and}\quad \Rr(\uu):=\begin{pmatrix}
        R_{12}(u_1,u_2) \\
R_{21}(u_2,u_1)
 \end{pmatrix} .
\end{align*}
\medskip

As previously, all linear operators (such as $\nabla,\partial_\tau \dots$) have to be understood line by line in the previous expressions. Any dot product $\langle \cdot,\cdot\rangle_{\E}$ on some space $\E$ of functions defined on $\Omega$ has to be understood as $\langle \uu,\cchi\rangle_{\E} = \langle u_1,\chi_1\rangle_{\E} + \langle u_2,\chi_2\rangle_{\E}$. The nonnegative symmetric matrix $A^\ep(u_1,u_2)$ will be denoted $A^\ep(\uu)$ and we 
also will often denote by $\Qq^\ep(\uu)$ the associated quadratic form. Finally we will use the obvious notation 
\begin{align*}
 \ffi_\ep(\uu)=\begin{pmatrix}
            \ffi_1^\ep(u_1)\\
	    \ffi_2^\ep(u_2)
           \end{pmatrix},
\end{align*}
in such a way that according to the previous notations, $\ww=\ffi_\ep(\uu)$ and $\ffi_\ep^{-1}(\ww)=\uu$.
\vspace{2mm}\\
We are going to define recursively solutions to a discrete-time problem,
 all lying in $\V_n$ ($n$ is fixed for the moment). 
Recall that the initial condition $\uu^0$ belongs to $\L^2(\Omega)^2$. Given $\uu^{k-1} \in \L^2(\Omega)^2$, we study the following problem (for fixed $\ep>0,\,\sigma\in[0,1]$, $k\geq 1$):   \vspace{2mm}\\
\begin{minipage}{1\textwidth}
\vspace{5mm}
\underline{\textsf{Problem $\textnormal{P}^\ep_\sigma(k,\uu^{k-1})$}}:\vspace{2mm}\\
\emph{Find $\ww^{k}\in\V_n^2$ such that [denoting $\uu^k=\ffi_\ep^{-1}(\ww^k)$] we have for all $\cchi\in\V_n^2$,}
\begin{align*}
 \sigma\Big[\left\langle\cchi,\partial_{\tau} \uu^k\right\rangle_{\L^2(\Omega)}+\left\langle\nabla \cchi, A^\ep(\uu^k) \nabla \ww^k\right\rangle_{\L^2(\Omega)}
-\left\langle\cchi,\Rr^\ep(\uu^{k})\right\rangle_{\L^2(\Omega)}\Big]=-\ep\left\langle\cchi,\ww^k\right\rangle_{\H^1(\Omega)}.
\end{align*}
\end{minipage}

\vspace{2mm}

Since $\ww^k\in\V_n\subset\L^{\infty}(\Omega)$, $\uu^k$ takes its values in some compact set of $]0,+\infty[^2$, hence the coefficients of $A^\ep(\uu^k)$ all belong to $ \L^{\infty}(\Omega)$, and this ensures that all the previous brackets are well-defined. Furthermore, $\uu^k$ clearly belongs to $\L^\infty(\Omega)^2 \subset \L^2(\Omega)^2$, so that the previous inductive definition is consistant.
\vspace{2mm}\\
The collection of problems $(\textnormal{P}^\ep_\sigma(k,\uu^{k-1}))_{k=1 \dots N}$ is hence a discrete-in-time version of the $\V_n$-weak form of \eqref{eq:symetric} in which an $\ep$-perturbation and a parameter $\sigma$ have been added. 
\bigskip

We end up this subsection by introducing (bearing in mind the notations of Definition \ref{defipseps}) the:

\begin{definition}\label{defie}
We introduce the entropy
(for some vector function $\uu:\Omega^2\rightarrow{]0, +\infty[}^2$):
\begin{align*}
 \mathscr{E}_\ep(\uu):=\sum_{i=1}^2\int_{\Omega}\psi_i^\ep(u_i)(x)\, \dd x.
\end{align*}
\end{definition}

\subsection{\emph{A priori} estimate} \label{sub32}

\begin{Propo}\label{propo:apriori1}
Under the assumptions of Theorem \ref{theo:theo} and keeping in mind the definitions of Subsection~\ref{sub31} (and Definition \ref{defipseps}), there exists a constant $K_T$ depending only on $T$ and the data of the equation ($a_{ij}$, $r_i$, $s_{ij}$) such that for any
 $\tau>0$, $\ep>0$ (small enough), and $\sigma\in[0,1]$ and any sequence (of length less than $N$)
 $(\ww^j)_{1\leq j\leq k}$,  with $\ww^j$ solving $\textnormal{P}_\sigma^\ep(j,\uu^{j-1})$, the following entropy inequality holds:
\begin{align}\label{ineq:propo:apriori1}
 \sigma\mathscr{E}_\ep(\uu^k)+\tau\sigma\sum_{j=1}^k\int_\Omega \Qq^\ep(\uu^j)(\nabla \ww^j)\,\dd x+\ep\tau\sum_{j=1}^k \|\ww^j\|_{\H^1(\Omega)}^2 \leq K_T(\mathscr{E}_\ep(\uu^0) +1).
\end{align}
\end{Propo}
\begin{proof}
 Plug $\ww^j$ in $\textnormal{P}_\sigma^\ep(j,\uu^{j-1})$ to obtain
\begin{align*}
 \sigma\int_\Omega \ww^j\cdot
\partial_{\tau} \uu^j\,\dd x
+\sigma\int_\Omega \, \Qq^\ep(\uu^j)(\nabla \ww^j)\dd x+\ep\|\ww^j\|_{\H^1(\Omega}^2
=\sigma\int_\Omega \ww^j\cdot\Rr^\ep(\uu^{j})\, \dd x.
\end{align*}
Using $(i)$ of Lemma \ref{lem:ineqconv:approx}, one gets, since $w_i={\psi_i^\ep}'(u_i)$,
\begin{align*}
  \int_\Omega \ww^j\cdot\partial_{\tau} \uu^j\,\dd x = \frac{1}{\tau} \int_\Omega \ww^j\cdot\big[\uu^j-\uu^{j-1}\big]\,\dd x \geq \frac{1}{\tau}(\mathscr{E}(\uu^j)-\mathscr{E}(\uu^{j-1})).
\end{align*}
For the reaction term, we have
\begin{align*}
 \int_\Omega \ww^j\cdot\Rr^\ep(\uu^{j})\, \dd x = \int_\Omega w^j_1 R_{12}^\ep(u_1^{j},u_2^{j})\, \dd x + \int_\Omega w^j_2 R_{21}^\ep(u_2^{j},u_1^{j})\, \dd x.
\end{align*}
For the sake of clarity, we avoid writing the superscript $j$ for a few lines. Let us focus on the first term of the right-hand side (the second one will be similar):
\begin{align*}
 \int_\Omega w_1 R_{12}^\ep(u_1,u_2)\, \dd x &= \int_\Omega {\psi_1^\ep}'(u_1) u_1\big(r_1-\gamma_\ep[s_{11}(u_1)+s_{12}(u_2)]\big)\, \dd x \\
&= r_1\int_\Omega {\psi_1^\ep}'(u_1) u_1 \dd x-\int_{\Omega}{\psi_1^\ep}'(u_1) u_1  \gamma_\ep[s_{11}(u_1)+s_{12}(u_2)]\,\dd x.
\end{align*}
From $(iv)$ of Lemma \ref{lem:ineqconv:approx}, we get the existence of a constant $D$ such that 
\begin{align*}
 x{\psi^\ep_1}'(x) \leq D(1+\ep)(1+\psi_1^\ep(x)),
\end{align*}
which implies
\begin{align*}
 r_1\int_\Omega {\psi_1^\ep}'(u_1) u_1 \dd x \leq  r_1D(1+\ep)\left( \mu(\Omega)+ \mathscr{E}_\ep(u)\right).
\end{align*}
For the two other terms, we use the fact that ${\psi_i^\ep}'\geq 0$ on $[1,+\infty[$, so that 
\begin{align*}
-\int_{\Omega}{\psi_1^\ep}'(u_1) u_1  \gamma_\ep[s_{11}(u_1)+s_{12}(u_2)]\,\dd x \leq -\int_{u_1\leq 1}{\psi_1^\ep}'(u_1) u_1  \gamma_\ep[s_{11}(u_1)+s_{12}(u_2)],
\end{align*}
\begin{align*}
-\int_{\Omega}{\psi_1^\ep}'(u_1) u_1  \gamma_\ep[s_{11}(u_1)+s_{12}(u_2)]\,\dd x \leq (-B+\frac{1}{e})\int_{\Omega} s_{11}(u_1)\dd x + (-B+\frac{1}{e})\int_{\Omega} s_{12}(u_2)\dd x.
\end{align*}
Finally, we use $s_{ij}(u)\leq K(1+u)$ to get
$$
-\int_{\Omega}{\psi_1^\ep}'(u_1) u_1  \gamma_\ep[s_{11}(u_1)+s_{12}(u_2)]\,\dd x\leq C\left(\mu(\Omega)+\int_\Omega [u_1+u_2]\,\right).
$$
By lemma~\ref{lem:ineqconv:approx} (iv), we have 
$$
\int_\Omega [u_1+u_2] \leq D(1+\ep)(\mu(\Omega)+\mathscr{E}_\ep(u)).
$$
 Finally we have, recovering the superscripts that we omitted before:
\begin{align}\label{ineq:beforegron}
 \frac{\sigma}{\tau}(\mathscr{E}_\ep(\uu^j)-\mathscr{E}_\ep(\uu^{j-1})) +\sigma\int_\Omega \, \Qq^\ep(\uu^j)(\nabla \ww^j)\dd x+ \ep\|\ww^j\|_{\H^1(\Omega}^2
\leq  \sigma K(1+\mathscr{E}_\ep(\uu^{j})),
\end{align}
for some constant $K$ independent of $j,n,\ep,\tau,\sigma$. Hence if one takes $\tau$ small enough (such that $1-\tau K\geq1/2$),
 one has in particular : 
\begin{align*}
 \mathscr{E}_\ep(\uu^j)\leq\mathscr{E}_\ep(\uu^{j-1})+ \tau K\mathscr{E}_\ep(\uu^{j}) + \tau K,
\end{align*}
which implies by Lemma \ref{lem:grondis}, with $\theta=C :=\tau K$,
\begin{align*}
 \mathscr{E}_\ep(\uu^j)\leq e^{2j\theta}\Big[\mathscr{E}_\ep(\uu^0)+1\Big]\leq e^{2j\tau K}(\mathscr{E}_\ep(\uu^0)+1)\leq e^{2TK}(\mathscr{E}_\ep(\uu^0)+1).
\end{align*}
Plugging this last inequality in \eqref{ineq:beforegron}, we get for some constant $C_T>0$,
\begin{align*}
 \sigma(\mathscr{E}_\ep(\uu^j)-\mathscr{E}_\ep(\uu^{j-1})) +\tau\sigma\int_\Omega \, \Qq^\ep(\uu^j)(\nabla \ww^j)\dd x +\tau\ep\|\ww^j\|_{\H^1(\Omega}^2
\leq  C_T\tau(1+\mathscr{E}_\ep(\uu^0)),
\end{align*}
which, after summation over $j\in\{1,\dots,k\}$ and using $k\tau \leq T$, ends the Proof of Proposition~\ref{propo:apriori1}.
\end{proof}

\subsection{Existence and estimates}\label{sub33}

The \emph{a priori} estimate proven in the previous subsection leads to the following proposition about existence:

\begin{Propo}\label{propo:exist}
Consider the assumptions of Theorem \ref{theo:theo} and recall the Definitions of Subsection \ref{sub31} (and Definition \ref{defipseps}).
\par
 For fixed $\tau (=T/N)$ (small enough), $n\in\N^*$, $\ww^0\in\V_n^2$ and $\ep>0$,
 there exists a sequence $(\ww^k)_{1\leq k\leq N}$ such that $\ww^k$ solves $\textnormal{P}_1^\ep(k,\uu^{k-1})$ for all $k\in\{1,\dots,N\}$,
and which furthermore satisfies the estimate (with $K_T$ depending only on $T$ and the data of the equation ($a_{ij}$, $r_i$, $s_{ij}$)):
\begin{align}\label{ineq:prepropo:exist}
\hspace{-1cm} \mathscr{E}_\ep(\uu^k)+\tau\sum_{j=1}^k\int_\Omega \Qq^\ep(\uu^k)(\nabla \ww^k)\,\dd x+\ep\tau\sum_{j=1}^k \|\ww^j\|_{\H^1(\Omega)}^2 &\leq K_T(\mathscr{E}_\ep(\uu^0) +1),\\
\label{ineq:prepropo:exist2}
\hspace{-1cm} \|\uu^k\|_{\L^1(\Omega)} &\leq \Big[\|\uu^0\|_{\L^1(\Omega)} +O(\sqrt{\ep})\Big]e^{r^\tau T}, \\
\hspace{-1cm} \label{ineq:prepropo:exist3}
\tau\sum_{j=1}^k\|\Rr^{-,\ep}(\uu^{j})\|_{\L^1(\Omega)} &\leq \Big[\|\uu^0\|_{\L^1(\Omega)} +O(\sqrt{\ep})\Big]\, (1 + T \,\rr\, e^{r^\tau T}),
\end{align}
where the sequence of positive scalars $(r^\tau)_\tau$ goes to $\rr=\max(r_1,r_2)$ when $\tau\rightarrow 0$.
\end{Propo}

\begin{proof}
\textit{Step 1: Estimates }\\
Let us  notice that the two last estimates \eqref{ineq:prepropo:exist2}--\eqref{ineq:prepropo:exist3} are essentially a
 discrete version of the classical formal $\L^1$~estimate obtained by integrating our system on $\Omega$.
 In order to recover this rigorously, we notice that since $\mathbf{1}_\Omega\in\V_n$, it is an admissible test function for the problem $\textnormal{P}_1^\ep(j,\uu^{j-1})$ for all $j \in\llbracket 1,k\rrbracket$, $ k \in\llbracket 1,N\rrbracket$, and we hence get
\begin{align}
\label{eq:tosum} \|\uu^j\|_{\L^1(\Omega)}- \|\uu^{j-1}\|_{\L^1(\Omega)}
+\ep\tau\int_\Omega\ww^j+\tau \|\Rr^{-,\ep}(\uu^{j})\|_{\L^1(\Omega)}&=\tau \|\Rr^+(\uu^{j})\|_{\L^1(\Omega)}.
\end{align}
It follows easily from the definition of $\Rr^+$ that the right-hand side of the previous inequality is not 
larger than $\tau \rr \|\uu^{j}\|_{\L^1(\Omega)}$. 
 Thus, 
\begin{align*}
 \|\uu^j\|_{\L^1(\Omega)} \leq  \|\uu^{j-1}\|_{\L^1(\Omega)} +\tau\rr \|\uu^j\|_{\L^1(\Omega)}-\ep\tau\int_\Omega 
[\ww_1^j+\ww^j_2].
\end{align*}
We notice then that,
$$
-\ep\tau\int_\Omega \ww_1^j+\ww_2^j=\sqrt{\ep}\tau \left(|\Omega|+\ep\|\ww^j\|_{L^2(\Omega)}^2\right).
$$
This leads to 
\begin{align*}
 \|\uu^j\|_{\L^1(\Omega)} \leq  \|\uu^{j-1}\|_{\L^1(\Omega)} +\tau\rr \|\uu^j\|_{\L^1(\Omega)}+\sqrt{\ep}\tau \left(|\Omega|+\ep\|\ww^j\|_{L^2(\Omega)}^2\right),
\end{align*}
which implies by Lemma \ref{lem:grondis} 
\begin{align*}
\|\uu^j\|_{\L^1(\Omega)} \leq \left[\|\uu^0\|_{\L^1(\Omega)}+\sqrt{\ep}\left( j\tau|\Omega|+\ep\tau\sum_{k=1}^j \|\ww^j\|_{L^2(\Omega)}^2\right)\right] \exp\left\{j\tau\frac{\rr}{1-\tau\rr}\right\}.
\end{align*}
Using the entropy estimate, one has
$$
\ep\tau\sum_{k=1}^j \|\ww^j\|_{L^2(\Omega)}^2\leq K_T(\mathscr{E}_\ep(\uu^0) +1)\qquad \text{ and }\qquad j\tau|\Omega|\leq |\Omega| T.
$$
Therefore, we obtain 
$$
 \|\uu^j\|_{\L^1(\Omega)}\leq \Big[\|\uu^0\|_{\L^1(\Omega)}+\sqrt{\ep}\left( T|\Omega|+ K_T(\mathscr{E}_\ep(\uu^0) +1)\right)\Big] e^{r^\tau T},
$$
with $r(\tau)$ having the mentioned property, that is precisely \eqref{ineq:prepropo:exist2}. 
\par
Quite similarly, we sum up in \eqref{eq:tosum} and obtain after using the bound on $\|\ww\|_1$:
$$
\|u^k\|_{L^1(\Omega)}+\tau \sum_{j=1}^k\|\Rr^{-,\ep}(\uu^{j})\|_{\L^1(\Omega)}=\tau \sum_{j=1}^k \|\Rr^+(\uu^{j})\|_{\L^1(\Omega)}+\sqrt{\ep}\left( T|\Omega|+ K_T(\mathscr{E}_\ep(\uu^0) +1)\right),
$$
and we conclude the estimate using the previous bound.
\vspace{2mm}\\
\textit{Step 2: Existence }

In the rest of the Proof,
 we will work in the finite dimensional Hilbert space $\E=\V_n^2$ with the dot product $\langle\cdot,\cdot\rangle_{\H^1(\Omega)}$ and the associated norm. We will proceed by induction and only prove the first iteration $\uu^0\in\L^2(\Omega)^2\Rightarrow \exists \,\ww^1 $ solving $\textnormal{P}_\sigma^\ep(1,\uu^{0})$, all the other
induction steps will be similar.

 First notice that the problem $\textnormal{P}_\sigma^\ep(1,\uu^{0})$ can be seen in the following way\vspace{2mm}\\
\begin{minipage}{1\textwidth}
\begin{center}
\medskip

Find $\ww^{1}\in\E$ such that 
$\forall \cchi\in \E,\,\sigma \L_{\uu^{0},\ww^1}(\cchi)=-\ep\left\langle\cchi,\ww^1\right\rangle_{\H^1(\Omega)}, $ 
\end{center}
\end{minipage}
\vspace{2mm}\\
where, for $\ww,\cchi \in \E$, $\L_{\uu^0,\ww}$ is defined by
\begin{align*}
 \L_{\uu^{0},\ww}(\cchi):=\frac{1}{\tau}\left\langle\cchi,\ffi_\ep^{-1}(\ww)-\uu^0\right\rangle_{\L^2(\Omega)}+\left\langle\nabla \cchi, A^\ep(\ffi_\ep^{-1}(\ww)) \nabla \ww\right\rangle_{\L^2(\Omega)}
-\left\langle\cchi,\Rr^\ep(\ffi_\ep^{-1}(\ww))\right\rangle_{\L^2(\Omega)}.
\end{align*}
As noticed before, for $\ww\in\E\subset\L^\infty(\Omega)^2$, $\ffi_\ep^{-1}(\ww)$ takes its values in some compact set of ${]0, +\infty[}^2$, so that the coefficients of $A^\ep(\ffi_\ep^{-1}(\ww))$ lie in $\L^\infty(\Omega)$, whence $\L_{\uu^{0},\ww^1}\in\V_n^\star$.\vspace{2mm}\\
Let us now define a map 
\begin{align*}
 T:[0,1]\times \E&\longrightarrow\E\\
 (\sigma,\vv)&\longmapsto \ww,
\end{align*}
such that:
\begin{align*}
 \forall \cchi\in \E\quad  \sigma \, \L_{\uu^0,\vv}(\cchi)=-\ep\left\langle\cchi,\ww\right\rangle_{\H^1(\Omega)}.
\end{align*}
Such a map is well-defined because of the usual representation theorem for finite dimensional spaces.

\begin{lem}
The map $T$ is continuous.
\end{lem}
\begin{proof}
Let $v^1\in E$. We want to prove that $T$ is continuous at $v^1$. 
 Since all norms are equivalent on $\E$, one can use the metric $|\cdot|+\|\cdot\|_{\L^{\infty}(\Omega)}$ on $[0,1]\times \E$ and $\|\cdot\|_{\H^1(\Omega)}$ on $\E$. In all what follows, we assume that $v^2$ belongs to the following open set of $E$
 $$
 \|v^1-v^2\|_\infty+\|v^1-v^2\|_{H^1}<1.
 $$
 Not that this means that $v^2$ belongs to a compact set of $E$ (just consider the inequalities as large)
 that we denote $K$. 
 In fact for $\|\cchi\|_{H^1}\leq 1$, one has
\begin{align*}
\Big|[\L_{\uu^0,\vv_1}-\L_{\uu^0,\vv_2}](\cchi)\Big| &\leq \frac{1}{\tau}\|\ffi_\ep^{-1}(\vv_1)-\ffi_\ep^{-1}(\vv_2)\|_{\L^2(\Omega)}\\
& +\|A^\ep(\ffi_\ep^{-1}(\vv_1)) \nabla \vv_1-A^\ep(\ffi_\ep^{-1}(\vv_2)) \nabla \vv_2\|_{\L^2\Omega)},\\
&+ \|\Rr^\ep(\ffi_\ep^{-1}(\vv_1))-\Rr^\ep(\ffi_\ep^{-1}(\vv_2))\|_{\L^2(\Omega)},
\end{align*}
and one can write for the second term of the right-hand side:
\begin{align*}
 \|A^\ep(\ffi_\ep^{-1}(\vv_1)) \nabla \vv_1-A^\ep(\ffi_\ep^{-1}(\vv_2)) \nabla \vv_2\|_{\L^2\Omega)} &\leq \|\big\{A^\ep(\ffi_\ep^{-1}(\vv_1))-A^\ep(\ffi_\ep^{-1}(\vv_2))\big\} \nabla \vv_1\|_{\L^2\Omega)} \\
&+\|A^\ep(\ffi_\ep^{-1}(\vv_2)) \big\{\nabla \vv_1- \nabla \vv_2\big\}\|_{\L^2\Omega)}.
\end{align*}

Since $\ffi_\ep^{-1}$, $A^\ep\circ\ffi_\ep^{-1}$ and $\Rr^\ep\circ\ffi_\ep^{-1}$ are all continuous, they are uniformly continuous (and bounded) on every compact of $\R^2$,
 and particularly on $K$. Let $\omega_K$ be a continuity modulus for all the previous functions on this compact. Hence, if $\vv_1,\vv_2$ take their values in $K$ and $\|\vv_1-\vv_2\|_{\L^{\infty}(\Omega)}\leq\eta$, we have
\begin{align*}
\Big|[\L_{\uu^0,\vv_1}-\L_{\uu^0,\vv_2}](\cchi)\Big| &\leq \frac{\omega_K(\eta)}{\tau}
+\omega_K(\eta)\mu(\Omega)^{1/2}\|\nabla \vv_1\|_{\L^{\infty}(\Omega)}\\
&+\|A^\ep(\ffi_\ep^{-1}(\vv_2))\|_{\L^{\infty}(\Omega)}\|\vv_1-\vv_2\|_{\H^1(\Omega)}+\omega_K(\eta),
\end{align*}
where we used the continuous injection $\L^{\infty}(\Omega)\hookrightarrow\L^2(\Omega)$. 
Since all norms are equivalent on $\E$,
\begin{align*}
\Big|[\L_{\uu^0,\vv_1}-\L_{\uu^0,\vv_2}](\cchi)\Big| \leq C_{K,\Omega}(\omega_K(\eta)+\eta),
\end{align*}
hence using the very definition of $T$ :
\begin{align*}
 \|T(\sigma,\vv_1)-T(\sigma,\vv_2)\|_{\H^1(\Omega)}=\frac{\sigma}{\ep}\sup_{\|\cchi\|_{H^1}\leq 1}\Big|[\L_{\uu^0,\vv_1}-\L_{\uu^0,\vv_2}](\cchi)\Big|\leq \frac{\sigma}{\ep}C_{K,\Omega}(\omega_K(\eta)+\eta),
\end{align*}
which gives the continuity w.r.t. the second variable. The continuity w.r.t. both variables
 is then straightforward.
 \end{proof}
\par
 It is now time to use Proposition \ref{propo:apriori1} to get an \emph{a priori} estimate on any fixed point of $T(\sigma,\cdot)$, for any $\sigma\in[0,1]$. In fact, the case $k=1$ of this Proposition exactly tells us that all this fixed points are in the ball of center $0$ and radius $K_T(\mathscr{E}(\uu^0)+1)/\ep\tau$. Since clearly $T(0,\cdot) \equiv 0$, we now can apply  
Theorem \ref{theo:brou} to see that $T(1,\cdot)$ has a fixed point, which is exactly the existence of $\ww^1$ and hence the first step of our induction machinery. Inequality \eqref{ineq:prepropo:exist} is then a direct consequence of \eqref{ineq:propo:apriori1} (with $\sigma=1$).
\end{proof}
The previous Proposition shows that 
(for fixed $\tau(=T/N)$ small enough, $n\in\N^*$, $\uu^0\in\L^2(\Omega)^2$, and $\ep>0$)
 there exists a sequence $(\ww^k)_{1\leq k\leq N}$ such that, for all $k\in\{1,\dots,N\}$, denoting $\uu^k:=\ffi_\ep^{-1}(\ww^k)$, we have for all $\cchi\in\V_n^2$:
\begin{align}\label{eq:propo:weakform}
\left\langle\cchi,\partial_{\tau} \uu^k\right\rangle_{\L^2(\Omega)}+\left\langle\nabla \cchi, A^\ep(\uu^k) \nabla \ww^k\right\rangle_{\L^2(\Omega)}+ \ep\left\langle\cchi,\ww^k\right\rangle_{\H^1(\Omega)} = \left\langle\cchi,\Rr^\ep(\uu^{k})\right\rangle_{\L^2(\Omega)},\end{align}
that we also can write ($i\neq j =1,2$)
\begin{align*} 
\partial_\tau \mathbb{P}_n u^k_i  -  \mathbb{P}_n \Delta\Big[ a_{ii}^\ep(u^k_i)+u_i^k a_{ij}(u_j^k)+ \ep u_i^ku_j^k\Big] +\ep w_i^k -\ep \Delta w_i^k = \mathbb{P}_n\Big[R_{ij}^\ep(u_i^{k},u_j^{k})\Big],
\end{align*}
where $\mathbb{P}_n$ is the $\L^2$-orthonormal projection on $\V_n$. 

\section{Asymptotic with respect to $n$}\label{section4}

 This Section is devoted to the passage to the limit in the space discretization ($n \to +\infty$), which can be
 summarized by the following Proposition:
 
 \begin{Propo}\label{propo:exist2}
Consider the assumptions of Theorem \ref{theo:theo}  and recall the notations of Definitions \ref{defipseps}, \ref{defie}, \ref{reactionpm} and \ref{timestep}. For fixed $\tau (=T/N)$ (small enough), and $\ep>0$,
 there exists a sequence $(\uu^k)_{1\leq k\leq N}$ of $\L^p(\Omega)^2$ for some $p>1$ (and a corresponding sequence $(\ww^k)_{1\leq k\leq N}$) such that for all  $\cchi \in \mathscr{C}^{\infty}_{\nu}(\overline{\Omega})^2$ 
(and for all $k\in\{1,\dots,N\}$),
 \begin{equation}\label{eq:weak_form}
\frac{1}{\tau}\langle \chi_i,u^{k}_i\rangle_{\L^2(\Omega)}-\frac{1}{\tau}\langle \chi_i,u^{k-1}_i\rangle_{\L^2(\Omega)} - \langle \Delta \chi_i,\Big[ a_{ii}^\ep(u^k_i)+u_i^k a_{ij} (u_j^k)+ \ep u_i^k u_j^k \Big]\rangle_{\L^2(\Omega)} 
\end{equation} 
$$+\ep \langle \chi_i- \Delta \chi_i,w_i^{k}\rangle_{\L^2(\Omega)} = \left\langle \chi_i,R_{ij}^\ep(u_i^{k},u_j^{k}) \right\rangle_{\L^2(\Omega)},$$
and which furthermore satisfies the estimates (with $K_T$ depending only on $T$ and the 
data of the equation ($a_{ij}$, $r_i$, $s_{ij}$)):
\begin{align}\label{ineq:afterlim1}
 \mathscr{E}_\ep(\uu^k)+\ep\tau\sum_{j=1}^k \|\ww^j\|_{\H^1(\Omega)}^2 \leq K_T(\mathscr{E}_\ep(\uu^0) +1);
\end{align} 
\begin{align}\label{ineq:afterlim2}
\tau\sum_{j=1}^k \int_\Omega  \left|\nabla \beta_\alpha(\uu^{j})\right|^2 \dd x \leq  K_T(\mathscr{E}_\ep(\uu^0) +1),
\end{align}
where $\beta_\alpha(x)= x^{\frac{1-\alpha}{2}}$;
\begin{align}
\label{ineq:afterlim3}
\hspace{-1cm} \|\uu^k\|_{\L^1(\Omega)} &\leq \|\uu^0\|_{\L^1(\Omega)} e^{T\, r_{\tau}}; \\
\hspace{-1cm} \label{ineq:afterlim4} \ep \tau \sum_{j=1}^k \|\ww^j\|_{\L^1(\Omega)} + \tau\sum_{j=1}^k\|\Rr^{-,\ep}(\uu^{j})\|_{\L^1(\Omega)} &\leq \|\uu^0\|_{\L^1(\Omega)}\Big[1 + T \rr e^{r(\tau) T}\Big].
\end{align}
\end{Propo}

\begin{proof}
We first recall the bounds that hold on $\uu^k_n$ (and $\ww^k_n$), whose existence is given by Proposition
\ref{propo:exist}.
\par 
 Using \eqref{ineq:propertyA1}, for all $k\in\{1,\dots,N\}$ and $n\in\N^*$, we have
\begin{align*}
\Qq^\ep(\uu^k_n)(\nabla \ww^k_n)  &\geq \frac{1}{u_1^{k,n}}{{a_{21}^\ep}'}(u_1^{k,n})|\nabla u_1^{k,n}|^2 +
 \frac{1}{u_2^{k,n}}{a_{12}^\ep}'(u_2^{k,n})|\nabla u_2^{k,n}|^2.
\end{align*}
Since $a_{ij}^\ep(x) = a_{ij}(x) + \ep x$, assumption \textnormal{\textbf{H2}} leads to
\begin{align}
\label{ineq:mino1}\hspace{-1.5cm}\Qq^\ep(\uu^k_n)(\nabla\ww^k_n) \geq \frac{4}{(1-\alpha)^2} \left|\nabla\beta_\alpha(u_1^{k,n})\right|^2  + \frac{4}{(1-\alpha)^2} \left|\nabla\beta_\alpha(u_2^{k,n})\right|^2 + 4\ep \left|\nabla\sqrt{u_1^{k,n}}\right|^2 + 4\ep \left|\nabla\sqrt{u_2^{k,n}}\right|^2,
\end{align}
(where we recall that $\beta_\alpha(x):=x^{\frac{1-\alpha}{2}}$).
We also have, because of \eqref{ineq:prepropo:exist2}, the boundedness of $(\uu_n^k)_{n\in\N^*}$ in $\L^1(\Omega)$. 
Since the asymptotics that
 we are studying is only w.r.t. $n$ (that is, $\var$ and $\tau$ are fixed), we see finally that $\left(\sqrt{u_1^{k,n}}\right)_{n\in\N^*}$ and $\left(\sqrt{u_2^{k,n}}\right)_{n\in\N^*}$ are bounded in $\H^1(\Omega)\hookrightarrow \L^{2^\star}(\Omega)$, with $2^\star = 2d/(d-2)>2$, and hence $(\uu_n^k)_{n\in\N^*}$ is eventually bounded in some $\L^p(\Omega)$ space with $1<p<\infty$.\vspace{2mm}\\
On the other hand, estimate \eqref{ineq:propertyA2} gives us 
\begin{align*}
\Qq^\ep(\uu_n^k)(\ww_n^k) \geq 4 \left|\nabla  \sqrt{a_{21}^\ep(u_1^{k,n})a_{12}^\ep(u_2^{k,n})}\right|^2,
\end{align*}
which together with \eqref{ineq:prepropo:exist} leads to
\begin{align*}
\left\|\nabla \sqrt{a_{21}^\ep(u_1^{k,n})a_{12}^\ep(u_2^{k,n})}\right\|_{\L^2(\Omega)}^2 \leq \frac{1}{4\tau}K_T(\mathscr{E}_\ep(\uu^0)+1).
\end{align*}
Using Poincar\'e-Wirtinger inequality, we get, for some constant $C_\Omega$ depending only on $\Omega$ 
\begin{align*}
\left\| \sqrt{a_{21}^\ep(u_1^{k,n})a_{12}^\ep(u_2^{k,n})}\right\|_{\L^2(\Omega)} &\leq C_\Omega \left\| \sqrt{a_{21}^\ep(u_1^{k,n})a_{12}^\ep(u_2^{k,n})}\right\|_{\L^1(\Omega)} + \frac{C_\Omega}{2\sqrt{\tau}}\sqrt{K_T(\mathscr{E}_\ep(\uu^0)+1)}\\
&\leq C_\Omega \sqrt{\|a_{21}^\ep(u_1^{k,n})\|_{\L^1(\Omega)}\|a_{12}^\ep(u_2^{k,n})\|_{\L^1(\Omega)}}+\frac{C_\Omega}{2\sqrt{\tau}}\sqrt{K_T(\mathscr{E}_\ep(\uu^0)+1)}\\
&\leq C_\Omega \Big[\|a_{21}^\ep(u_1^{k,n})\|_{\L^1(\Omega)}+\|a_{12}^\ep(u_2^{k,n})\|_{\L^1(\Omega)}\Big]+\frac{C_\Omega}{2\sqrt{\tau}}\sqrt{K_T(\mathscr{E}_\ep(\uu^0)+1)},
\end{align*}
and, again because point $(iv)$ of Lemma \ref{lem:ineqconv:approx}, for some constant $D$ ($\ep<1$),
\begin{align*}
a_{ji}^\ep(u_i^{k,n}) \leq D\,(2+ \psi_i^\ep(u_i^{k,n})),
\end{align*}
so we finally have
\begin{align*}
\left\| \sqrt{a_{21}^\ep(u_1^{k,n})a_{12}^\ep(u_2^{k,n})}\right\|_{\L^2(\Omega)} \leq C_\Omega\Big[ 8 D\mu(\Omega)+2D\mathscr{E}_\ep(\uu^k_n)\Big]+\frac{1}{2\sqrt{\tau}}\sqrt{K_T(\mathscr{E}_\ep(\uu^0)+1)},
\end{align*}
and hence thanks to (\ref{ineq:prepropo:exist}), the sequence
$\left(\sqrt{a_{21}^\ep(u_1^{k,n})a_{12}^\ep(u_2^{k,n})}\right)_{n\in\N^*}$ is bounded in $\H^1(\Omega)\hookrightarrow \L^{2^\star}(\Omega)$, with $2^\star=2d/(d-2)>2$. The previous continuous injection implies then that $(a_{21}^\ep(u_1^{k,n})a_{12}^\ep(u_2^{k,n}))_{n\in\N^*}$ is bounded in some $\L^p(\Omega)$ space with $1<p<\infty$ and since $a_{ji}^\ep(u_i^{k,n})=a_{ji}(u_i^{k,n})+\ep u^{k,n}_i$, we eventually get that $(u_1^{k,n}u_2^{k,n})_{n\in\N^*}$, $\left(u_1^{k,n}a_{12}(u_2^{k,n})\right)_{n\in\N^*}$ and $\left(u_2^{k,n}a_{21}(u_1^{k,n})\right)_{n\in\N^*}$ are bounded in the same $\L^p(\Omega)$. 
\vspace{2mm}\\
Summing the bounds already obtained [we recall that they hold for a given $\var$ and $\tau$], we see that (for all $k\in\{1,\dots,N\}$):
\begin{itemize}
\item[$\bullet$] $\left(\sqrt{u_1^{k,n}}\right)_{n\in\N^*}$ and  $\left(\sqrt{u_2^{k,n}}\right)_{n\in\N^*}$ are bounded in $\H^1(\Omega)$,
\item[$\bullet$] $\left(\sqrt{a_{21}^\ep(u_1^{k,n})a_{12}^\ep(u_2^{k,n})}\right)_{n\in\N^*}$ is bounded in $\H^1(\Omega)$,
\item[$\bullet$] $(u_1^{k,n})_{n\in\N^*}$, $(u_2^{k,n})_{n\in\N^*}$, $(u_1^{k,n}u_2^{k,n})_{n\in\N^*}$, $\left(u_1^{k,n}a_{12}(u_2^{k,n})\right)_{n\in\N^*}$ and $\left(u_2^{k,n}a_{21}(u_1^{k,n})\right)_{n\in\N^*}$ are bounded in some $\L^p(\Omega)$ space, with $1<p<\infty$,
\item[$\bullet$] and obviously $(\ww_n^k)_{n\in\N^*}$ is bounded in $\H^1(\Omega)$, because of estimate \eqref{ineq:prepropo:exist}.
\end{itemize}
\bigskip

Since we are only dealing with a finite number of values for $k\in\{1,\dots,N\}$ (at this point, the functions are not time-depending), we shall only detail the study of $(\uu^1_n)_{n\in\N^*}$, the other values of $k$ being similar (one only has to extract a finite number of subsequences). For every test function $\cchi=(\chi_1,\chi_2)\in\V_n^2$, the weak formulation \eqref{eq:propo:weakform} may be written ($i\neq j \in\{1,2\}$)
\begin{align*}
\frac{1}{\tau}\langle \chi_i,u^{1,n}_i\rangle_{\L^2(\Omega)}-\frac{1}{\tau}\langle \chi_i,u^{0,n}_i\rangle_{\L^2(\Omega)} &- \langle \Delta \chi_i,\Big[ a_{ii}^\ep(u^{1,n}_i)+u_i^{1,n} a_{ij}(u_j^{1,n})+ \ep u_i^{1,n}u_j^{1,n}\Big]\rangle_{\L^2(\Omega)} \\
&+\var \langle \chi_i- \Delta \chi_i,w_i^{1,n}\rangle_{\L^2(\Omega)} = \left\langle \chi_i,R_{ij}^\ep(u_i^{1,n},u_j^{1,n}) \right\rangle.
\end{align*}
First we extract (but do not change the indexes) a subsequence of $(\ww^1_n)_{n\in\N^*}$ converging in $\L^2(\Omega)$ and almost everywhere to some element $\ww^1\in\H^1(\Omega)^2$. The almost everywhere convergence is transmitted to $(\uu^1_n)_{n\in\N^*}$, since $\uu^1_n=\ffi_{\ep}^{-1}(\ww^1_n)$, the $\ffi_i^\ep$ functions being homeomorphisms.
 Because of the $\L^p(\Omega)$ ($1<p<\infty$) bound for
 $(\uu^1_n)_{n\in\N^*}$, we can extract a subsequence (of the previous subsequence) converging weakly in $\L^p(\Omega)$ to some function $\vv^1\in\L^1(\Omega)^2$. It follows by a classical argument that $\vv^1$ is almost everywhere equal to $\uu^1:=\ffi_\ep^{-1}(\ww^1)$, that is the almost everywhere limit of $(\uu^1_n)_{n\in\N^*}$. 
The same argument holds for $(u_1^{1,n}u_2^{1,n})_{n\in\N^*}$ and $\left(u_i^{1,n}a_{ij}(u_j^{1,n})\right)_{n\in\N^*}$, $i\neq j=1,2$, since we have the same $\L^p(\Omega)$ bound. We may also assume\footnote{In fact we have $\uu^{0,n} = \uu^0$ ! This line is just to justify the handling of the same term in the other time steps.} that $(\uu^{0,n})_{n\in\N^*}$ converges weakly in $\L^p(\Omega)$ to some function $\uu^0\in\L^p(\Omega)$.  
Eventually, the cutoff perturbation introduced on $a_{ii}^\ep$ and the (superlinear) reaction terms ensures the weak convergence, in $\L^p(\Omega)$, of  $\left(a_{ii}^\ep(u_i^{1,n})\right)_{n\in\N^*}$ and $(\Rr^\ep(\uu^{1,n}))_{n\in\N*}$, respectively to $a_{ii}^\ep(u_i^1)$ and $\Rr^\ep(\uu^{1})$.\vspace{2mm}\\
We now fix a test function  $\cchi$ in $\displaystyle \cup_{n\in\N^*} \V_n^2$, say $\cchi\in\V_m^2$, $m\in\N^*$. Then, for $n$ greater than $m$, since the sequence of spaces $\V_n$ is increasing, we have the previous weak formulation that passes to the limit thanks to all the 
extractions that we made,
 and get eq. (\ref{eq:weak_form}).
\bigskip

It remains to prove the bounds announced in the Proposition. 



Since we have almost everywhere convergence of $(\uu^{k}_n)_{n\in\N}$ and weak $\H^1(\Omega)$ convergence of $(\ww^{k}_n)_{n\in\N}$, Fatou's Lemma and the classical estimate for the weak limits give directly, for all $k\in\llbracket 1 , N \rrbracket $,
estimate (\ref{ineq:afterlim1}).
 
We use then \eqref{ineq:mino1} that we rewrite here:
\begin{align*}
\hspace{-1cm}\Qq^\ep(\uu^k_n)(\nabla \ww^k_n) \geq \frac{4}{(1-\alpha)^2} \left|\nabla\beta_\alpha(u_1^{k,n})\right|^2  + \frac{4}{(1-\alpha)^2} \left|\nabla \beta_\alpha(u_2^{k,n})\right|^2 + 4\ep \left|\nabla \sqrt{u_1^{k,n}}\right|^2 + 4\ep \left|\nabla \sqrt{u_2^{k,n}}\right|^2,
\end{align*}
where we recall that $\beta_\alpha(x)= x^{\frac{1-\alpha}{2}}$.
We notice that for all  $j\in\llbracket 1,k\rrbracket$, the sequence $(\beta_\alpha(\uu^j_n))_{n\in\N}$ was, precisely in view of the previous inequality, bounded in $\H^1(\Omega)$, so that we can (adding another extraction) assume that we also had the convergence  $(\beta_\alpha(\uu^k_n))_{n\in\N}\rightharpoonup  \beta_\alpha(\uu^k)$ in $\H^1(\Omega)$ (using the uniqueness of the weak limit) and we hence have estimate \eqref{ineq:afterlim2} using Fatou's Lemma. 

As for \eqref{ineq:prepropo:exist2} and \eqref{ineq:prepropo:exist3}, we notice that we had a $\L^p(\Omega)$ ($p>1$) bound for both $(\uu^{j,n})_{n\in\N^*}$ and $(\Rr^{-,\ep}(\uu^{j,n}))_{n\in\N^*}$, associated with almost everywhere convergence. This is sufficient (using Egoroff's theorem) to get the (strong) convergence of these two sequences in $\L^1(\Omega)$. 
Estimates \eqref{ineq:prepropo:exist2} and \eqref{ineq:prepropo:exist3} give hence \eqref{ineq:afterlim3}, and \eqref{ineq:afterlim4}
 after taking the limit in $n$.
\end{proof}

\section{Duality Estimate}
\label{section5}

\subsection{Notations}\label{sub51}


\begin{definition}\label{def:pro}
For a given family $h:=(h^k)_{1\leq k \leq N}$ of functions defined on $\Omega$, we denote by $\underline{h}^{\boldsymbol{\tau}}$ the step (in time) function defined on $]0,T] \times \Omega$ by
\begin{align*}
\underline{h}^{\boldsymbol{\tau}}(t,x):= \sum_{k=1}^N h^k(x) \mathbf{1}_{](k-1)\tau,k\tau]} (t).
\end{align*}
We then have by definition, for all $p,q\in[1,\infty[$,
\begin{align*}
\|\underline{h}^{\boldsymbol{\tau}}\|_{\L^q\big([0,T];\L^p(\Omega)\big)}=\left(\sum_{k=1}^N \tau \|h^{k}\|^q_{\L^p(\Omega)}\right)^{1/q},
\end{align*}
and in particular
\begin{align*}
\|\underline{h}^{\boldsymbol{\tau}}\|_{\L^p(Q_T)}=\left(\sum_{k=1}^N \tau \int_\Omega |h^{k}(x)|^p\dd x\right)^{1/p}. 
\end{align*}
\end{definition}


\subsection{Duality estimate: abstract result} \label{sub52}

This subsection is devoted to the establishment of a discretized version of the duality estimates devised for singular parabolic equations
in \cite{PiSc}, \cite{MaPi}. We start with the following Lemma:

\begin{lem}\label{propo:paragene}
Consider a smooth function $b \in\mathscr{C}^\infty(\Omega)$ such that $b \geq \gamma >0$ for some constant $\gamma$.
 For a given $\Psi\in\L^2(\Omega)$, the variational problem associated to the equation
\begin{align}\label{star}
- \Phi + b \Delta \Phi = \Psi
\end{align}
is well-posed in $\H^2_\nu(\Omega)$ (with $\L^2(\Omega)$ test functions). If $\Psi$ is assumed to be nonpositive, 
then the solution of the previous equation is nonnegative.
\end{lem}

\begin{proof}
Since $b$ is bounded from below (by $\gamma$) and above, we can first (uniquely) solve the (variationnal formulation of the) problem 
\begin{align}\label{starstar}
-b^{-1}\Phi+\Delta \Phi = b^{-1}\Psi,
\end{align} in the Hilbert space
 $\H^1(\Omega)$, using Lax-Milgram Theorem.
 \par
Elliptic regularity ensures that the constructed solution lies in fact in $\H^2(\Omega)$, and even in $\H^2_\nu(\Omega)$ because of the variationnal formulation,  so that  \eqref{starstar}  is satisfied almost everywhere and we recover a solution of \eqref{star} after a multiplication by $b$. Uniqueness is straightforward. All  $\L^2(\Omega)$ test functions are obtained by density.
The last part of the Lemma is obtained by maximum principle.
\end{proof}
\vspace{2mm}

We then turn to the:

\begin{lem}\label{lem:dual_classic}
Consider a real number $r>0$ such that $1-2r\tau>0$ and two families $b:=(b^k)_{1\leq k \leq N}$, $F:=(F^k)_{1\leq k \leq N}$ of  $\mathscr{C}^\infty(\Omega)$ functions. Assume that for all $k\in\llbracket 1 ,N \rrbracket$, $b^k \geq 1$ and $F^k \leq 0$ (pointwise).
\par
Then there exists a family ${\Phi:=(\Phi^k)_{1\leq k \leq N}\in\H^2_\nu(\Omega)^N}$ of nonnegative functions such that,
 (defining $\Phi^{N+1}:=0$), 
\begin{align}\label{eq:kth}
\forall k\in\llbracket 1 ,N\rrbracket, \quad\frac{\Phi^{k+1}-\Phi^{k}}{\tau}+b^{k}\Delta\Phi^{k}=\sqrt{b^{k}}F^{k}-r\Phi^{k},
\end{align}
where eq. (\ref{eq:kth}) has to be understood weakly, against $\L^2(\Omega)$ test functions. 
\par
This family satisfies furthermore 
\begin{align}
\label{ineq:duality_bound1} \forall j\in\llbracket 1, N\rrbracket,\quad  \|\nabla \Phi^{j}\|_{\L^2(\Omega)} &\leq e^{r(\tau) T}\big\|\underline{F}^{\boldsymbol{\tau}}\big\|_{\L^2(Q_T)},\\
\label{ineq:duality_bound2} \big\|\sqrt{\underline{b}^{\boldsymbol{\tau}}}\Delta \underline{\Phi}^{\boldsymbol{\tau}}\big\|_{\L^2(Q_T)} &\leq e^{r(\tau) T}\big\|\underline{F}^{\boldsymbol{\tau}}\big\|_{\L^2(Q_T)},\\
\label{ineq:duality_bound3} \forall j\in\llbracket 1, N\rrbracket,\quad \|\Phi^j\|_{\H^1(\Omega)}& \leq e^{r(\tau) T}C_\Omega\Big[e^{r(\tau) T}+\big\|\sqrt{\underline{b}^{\boldsymbol{\tau}}}\big\|_{\L^2(Q_T)}\Big]\big\|\underline{F}^{\boldsymbol{\tau}}\big\|_{\L^2(Q_T)},
\end{align}
where the sequence of positive scalars $(r(\tau))_\tau$ goes to $r$ when $\tau\rightarrow 0$.
\end{lem}

\begin{proof}
A straightforward induction using Lemma \ref{propo:paragene} gives us step by step the existence of the family $(\Phi^k)_{1\leq k \leq N}$ and the nonnegativity of its elements. Now in the $k$-th equation, $\tau \Delta \Phi^{k}\in\L^2(\Omega)$ is an admissible test function and we have hence after integration by parts:
\begin{align*}
\int_\Omega |\nabla \Phi^{k}|^2\dd x - \int_\Omega \nabla\Phi^{k+1}\cdot \nabla\Phi^{k} \dd x+\tau\int_\Omega b^{k}|\Delta \Phi^{k}|^2\dd x=\tau\int_\Omega \sqrt{b^{k}}F^{k}\Delta\Phi^{k}\dd x + r\tau\int_\Omega |\nabla \Phi^{k}|^2\dd x.
\end{align*}
Using
\begin{align*}
 \sqrt{b^{k}}F^{k}\Delta\Phi^{k}\leq \frac{|F^{k}|^2}{2}+\frac{b^{k}|\Delta\Phi^{k}|^2}{2}
\end{align*}
and
\begin{align*}
|\nabla\Phi^{k}|^2- \nabla\Phi^{k+1}\cdot\nabla\Phi^{k} \geq \frac{1}{2} \left(|\nabla\Phi^{k}|^2-|\nabla\Phi^{k+1}|^2\right),
\end{align*}
we get, for all $k\in\llbracket 1,N\rrbracket$
\begin{align*}
\int_\Omega |\nabla\Phi^{k}|^2\dd x-\int_\Omega |\nabla\Phi^{k+1}|^2\dd x+\tau\int_\Omega b^{k}|\Delta\Phi^{k}|^2\dd x\leq \tau\int_\Omega |F^{k}|^2\dd x+2r\tau\int_\Omega |\nabla \Phi^{k}|^2\dd x.
\end{align*}
Let us introduce the auxiliary sequences ($1\leq k\leq N$):
\begin{align*}
\Psi^k&:=\frac{\Phi^{k}}{(1-2r\tau)^{k/2}},\\
G^k&:=\frac{F^{k}}{(1-2r\tau)^{k/2}},
\end{align*}
we have then, dividing the previous inequality by $(1-2r\tau)^{k+1}$,
\begin{align*}
\frac{1}{1-2r\tau}\int_\Omega |\nabla\Psi^{k}|^2\dd x-\int_\Omega |\nabla\Psi^{k+1}|^2\dd x+\frac{\tau}{1-2r\tau}\int_\Omega b^{k}|\Delta\Psi^{k}|^2\dd x&\leq \frac{\tau}{1-2r\tau}\int_\Omega |G^{k}|^2\dd x\\
&+\frac{2r\tau}{1-2r\tau}\int_\Omega |\nabla \Psi^{k}|^2\dd x,
\end{align*}
hence 
\begin{align*}
\int_\Omega |\nabla\Psi^{k}|^2\dd x-\int_\Omega |\nabla\Psi^{k+1}|^2\dd x+\frac{\tau}{1-2r\tau}\int_\Omega b^{k}|\Delta\Psi^{k}|^2\dd x\leq \frac{\tau}{1-2r\tau}\int_\Omega |G^{k}|^2\dd x.
\end{align*}
Now for $j\in\llbracket 1,N\rrbracket $, summing up over $j\leq k \leq N$ and using $\Phi^{N+1}=0$, we get
\begin{align*}
\int_\Omega |\nabla\Psi^{j}|^2\dd x+\frac{\tau}{1-2r\tau}\sum_{k=j}^{N}\int_\Omega b^{k}|\Delta\Psi^{k}|^2\dd x\leq \frac{\tau}{1-2r\tau} \sum_{k=1}^{N}\int_\Omega |G^{k}|^2 \dd x, 
\end{align*}
and hence, in terms of the $\Phi^k, F^k$,
\begin{align*}
\int_\Omega |\nabla\Phi^{j}|^2\dd x+\tau\sum_{k=j}^{N}\int_\Omega b^{k}|\Delta\Phi^{k}|^2\dd x&\leq \frac{\tau}{(1-2r\tau)^{N+1}} \sum_{k=1}^{N}\int_\Omega |F^{k}|^2 \dd x\\
&\leq \exp\left\{\frac{(N+1)2r\tau}{1-2r\tau}\right\} \sum_{k=1}^{N}\int_\Omega \tau|F^{k}|^2 \dd x,
\end{align*}
which gives immediately \eqref{ineq:duality_bound1} and \eqref{ineq:duality_bound2} in the particular case $j=1$.\vspace{2mm}\\
Now for the last bound, notice that after integrating the $k$-th equation, we have 
\begin{align*}
(1-r\tau)\int_\Omega \Phi^k \dd x = \int_\Omega \Phi^{k+1} \dd x + \tau \int_\Omega \big\{b^k\Delta\Phi^k -\sqrt{b^k}F^k \big\}\, \dd x.
\end{align*}
Defining this time 
\begin{align*}
\Psi^k&:=\frac{\Phi^{k}}{(1-r\tau)^{k}},\\
G^k&:=\frac{F^{k}}{(1-r\tau)^{k}},
\end{align*}
we get, dividing the previous equality by $(1-r\tau)^{k+1}$,
\begin{align*}
\int_\Omega \Psi^k \dd x  = \int_\Omega \Psi^{k+1} \dd x + \frac{\tau}{1-r\tau} \int_\Omega \big\{b^k\Delta\Psi^k -\sqrt{b^k}G^k \big\} \dd x,
\end{align*}
which, after summing up all over $j\leq k\leq N$, leads to 
\begin{align*}
\int_\Omega \Phi^j \leq\int_\Omega \Psi^j \dd x &= \frac{1}{1-r\tau} \sum_{k=j}^N \tau \int_\Omega \big\{b^k\Delta\Psi^k -\sqrt{b^k}G^k \big\} \dd x\\
&\leq \frac{1}{(1-r\tau)^{N+1}}\sum_{k=j}^N \tau \int_\Omega \big\{b^k|\Delta\Phi^k| + \sqrt{b^k}|F^k| \big\} \dd x,
\end{align*}
and we have hence, using \eqref{ineq:duality_bound2},
\begin{align}
\label{ineq:dualnorm1}\left|\int_\Omega \Phi^j \dd x\right|\leq \Big[\big\|\sqrt{\underline{b}^{\boldsymbol{\tau}}}\big\|_{\L^2(Q_T)}+e^{r(\tau) T}\Big]\big\|\underline{F}^{\boldsymbol{\tau}}\big\|_{\L^2(Q_T)}e^{r(\tau) T},
\end{align}
so that \eqref{ineq:duality_bound3} follows from \eqref{ineq:duality_bound1} and Poincar\'e-Wirtinger inequality.\end{proof}
\vspace{2mm}

If we want to use functions of the previous family $\Phi$ as test functions in the weak formulation~\eqref{eq:weak_form}, we have to show that they actually belong to $\textnormal{W}^{2,p'}(\Omega)$. This can be done using the following lemma
\begin{lem}
Let $r$ be a strictly positive real number. If $w\in \H^2_\nu(\Omega)$ satisfies $w\geq 0$ and $-\Delta w\leq f+rw$ almost everywhere, for some $f$ in $\L^q(\Omega)$, $q>\max(d/2,2)$, then 
\begin{align*}
\|w\|_{\L^{\infty}(\Omega)} \leq C_{\Omega,q,r}\left(\|f\|_{\L^q(\Omega)}+ \|w\|_{\L^1(\Omega)}\right).
\end{align*}
\end{lem}
\begin{proof}
$C_{\Omega,q}$ and $C_q$ denotes constants that may vary from line to line. Lemma \ref{propo:paragene} gives us the existence of $g\in\H^2_\nu(\Omega)$, unique solution of 
\begin{align*}
g-\Delta g = (r+1)w + f \in\L^{\min(2^{\bullet},q)}(\Omega),
\end{align*} 
where $p^\bullet$ is defined by the Sobolev embedding $\W^{2,p}(\Omega)\hookrightarrow\L^{p^{\bullet}}(\Omega)$. If $2^\bullet < q$, then the elliptic regularity of $(\Id-\Delta)$ 
 and a Sobolev embedding (cf. \cite{gilbarg}) ensures
\begin{align*}
\|g\|_{\L^{{2^{\bullet}}^\bullet}}(\Omega) \leq C_\Omega \|g\|_{\W^{2,2^\bullet}(\Omega)} \leq C_\Omega\|(r+1)w + f \|_{\L^{2^\bullet}(\Omega)}.
\end{align*}
But we also have  
\begin{align*}
w-\Delta w\leq (r+1)w+f,
\end{align*}
that we can write
\begin{align*}
(w-g) -\Delta(w-g) \leq 0,
\end{align*}
so that by the weak maximum principle in $\H^1_\nu(\Omega)$, we have 
\begin{align*}
0 \leq w \leq g \in \L^{{2^\bullet}^\bullet}(\Omega).
\end{align*}
 The previous argument can be iterated in order to finally get $w\in\L^q(\Omega)$,
 and thus using the elliptic regularity mentioned before,
\begin{align*}
\|g\|_{\W^{2,q}(\Omega)} \leq C_\Omega \|(r+1)w + f\|_{\L^q(\Omega)}.
\end{align*}
Since $q>d/2$, we get the $\L^\infty$ estimate
\begin{align*}
\|g\|_{\L^\infty(\Omega)} \leq C_{\Omega,q} \|(r+1)w + f\|_{\L^q(\Omega)},
\end{align*}
that is again transmitted to $w$, thanks to the weak maximum principle :
\begin{align*}
\|w\|_{\L^\infty(\Omega)} \leq C_{\Omega,q}\|f\|_{\L^q(\Omega)}+C_{\Omega,q}(r+1)\|w\|_{\L^q(\Omega)}.
\end{align*}
Since
\begin{align*}
\|w\|_{\L^q(\Omega)}\leq \|w\|_{\L^1(\Omega)}^{1/q}\|w\|_{\L^\infty(\Omega)}^{1/q'},
\end{align*}
we get by Young's inequality
\begin{align*}
C_{\Omega,q}(r+1)\|w\|_{\L^q(\Omega)}\leq C_{q,r} \|w\|_{\L^1(\Omega)}+\frac{1}{2}\|w\|_{\L^\infty(\Omega)},
\end{align*}
and we are eventually able to conclude that
\begin{align*}
\|w\|_{\L^{\infty}(\Omega)} \leq C_{\Omega,q,r}\left(\|f\|_{\L^q(\Omega)}+ \|w\|_{\L^1(\Omega)}\right).
\end{align*}
\end{proof}
\vspace{2mm}

We  now can prove the following:
\begin{lem}\label{lem:dual_classic2}
Under the assumptions of Lemma \ref{lem:dual_classic}, the sequence of functions $\Phi$ satisfies in fact, for all $k\in\llbracket 1,N\rrbracket$
\begin{align*}
\| \Phi^k \|_{\L^\infty(\Omega)} + \| \Delta \Phi^k \|_{\L^\infty(\Omega)} &\leq C_{\tau,\Omega} \Big[1+\big\|\sqrt{\underline{b}^{\boldsymbol{\tau}}}\big\|_{\L^2(Q_T)}\Big]\big\|\underline{F}^{\boldsymbol{\tau}}\big\|_{\L^\infty(Q_T)},
\end{align*}
the constant $C_{\tau,\Omega}$ depending only on $\tau$ and $\Omega$, but (severely) blowing up as $\tau\rightarrow 0$. In particular we get that $\Phi^k\in\W^{2,q}(\Omega)$ for all $q\in[1,\infty[$.
\end{lem}
\begin{proof}
Since $\Phi^N\geq 0$, $\Phi^{N+1}=0$, $F^N \leq 0$ and $b^N \geq 1$, the $N$-th equation of \eqref{eq:kth} gives us
\begin{align*}
-\Delta \Phi^N \leq -F^{N}+r\Phi^N,
\end{align*}
and since $F^N \in \L^\infty(\Omega)\hookrightarrow \L^s(\Omega)$ for all $s\geq 1$, the previous Lemma applies:
\begin{align*}
  \|\Phi^N\|_{\L^\infty(\Omega)} \leq C_{\Omega,r} \Big[\big\|\underline{F}^{\boldsymbol{\tau}}\big\|_{\L^\infty(Q_T)} +   \|\Phi^N\|_{\L^1(\Omega)}\Big],
\end{align*}
and we use then \eqref{ineq:dualnorm1} to get 
\begin{align*}
  \|\Phi^N\|_{\L^\infty(\Omega)} \leq C_{\tau,\Omega, T} \Big[1+\big\|\sqrt{\underline{b}^{\boldsymbol{\tau}}}\big\|_{\L^2(Q_T)}\Big]\big\|\underline{F}^{\boldsymbol{\tau}}\big\|_{\L^\infty(Q_T)}.
\end{align*}
For all $k\in \llbracket 1,N\rrbracket$, using the $k$-th equation of \eqref{eq:kth},
 we have in the same way
\begin{align*}
-\Delta \Phi^k \leq \Phi^{k+1}-F^{k}+r\Phi^k,
\end{align*}
and hence a descending induction gives eventually ($C_{\tau,\Omega}$ varies from line to line)
\begin{align*}
\| \Phi^k \|_{\L^\infty(\Omega)} &\leq C_{\tau,\Omega} \Big[1+\big\|\sqrt{\underline{b}^{\boldsymbol{\tau}}}\big\|_{\L^2(Q_T)}\Big]\big\|\underline{F}^{\boldsymbol{\tau}}\big\|_{\L^\infty(Q_T)}.
\end{align*}
To handle the laplacian terms, we write for all $k\in\llbracket 1,N\rrbracket$, using again \eqref{eq:kth},
\begin{align*}
-\Delta \Phi^{k}=\frac{\Phi^{k+1}-\Phi^k}{b^{k}\tau}-\frac{F^{k}}{\sqrt{b^{k}}}  + r\frac{\Phi^k}{b^k},
\end{align*}
which directly gives, using the previous estimate,
\begin{align*}
\| \Delta \Phi^k \|_{\L^\infty(\Omega)} &\leq C_{\tau,\Omega} \Big[1+\big\|\sqrt{\underline{b}^{\boldsymbol{\tau}}}\big\|_{\L^2(Q_T)}\Big]\big\|\underline{F}^{\boldsymbol{\tau}}\big\|_{\L^\infty(Q_T)}.
\end{align*}
In particular, the norms $\| \Phi^k-\Delta \Phi^k \|_{\L^q(\Omega)}$ are all finite for $q\in]1,\infty[$. Thereby, using again the aforementioned (see \cite{gilbarg}) result on elliptic regularity, we get the finiteness of all the norms $\|\Phi^k\|_{\W^{2,q}(\Omega)}$ for all $q<\infty$.
\end{proof}
 
\subsection{Duality estimate: application to the system} \label{sub53}

\begin{Propo}\label{propo:estimdualite}
Under the assumptions of Theorem \ref{theo:theo} and using the Definitions \ref{defipseps}, \ref{timestep} and \ref{def:pro}, any weak solution of eq. (\ref{eq:weak_form}) satisfies the following estimate:
\begin{small}
\begin{align}\label{ineq:dual} 
\hspace{-1.6cm}\left\|\underline{u}_{\hspace{1pt} 1}^{\boldsymbol{\tau}}\sqrt{d_{11}^\ep\big({\underline{u}_{\hspace{1pt} 1}^{\boldsymbol{\tau}}}\big)+a_{12}^\ep\big({\underline{u}_{\hspace{1pt} 2}^{\boldsymbol{\tau}}}\big)}\right\|_{\L^2(Q_T)}\leq D,
\end{align}
\end{small}
where $D$ is a constant depending only on $T,\Omega$ and $\uu^0$.
 The same estimate holds when the subscripts $1$ and $2$ are exchanged.
\end{Propo}

\begin{proof}
We shall  denote by $\underline{u}_1^{\boldsymbol{\tau}}$ and ${\underline{u}_{\hspace{1pt} 2}^{\boldsymbol{\tau}}}$ the step (in time) functions (see Definition \ref{def:pro}) associated to the two families $(u_1^k)_{1\leq k \leq N}$ and $(u_2^k)_{1\leq k \leq N}$ (excluding hence $u_1^0$ and $u_2^0$). Since the original system is symmetric in $u_1$ and $u_2$, let us focus on $u_1$ for now. As in Lemma \ref{lem:dual_classic}, we
consider two families $b:=(b^k)_{1\leq k \leq N}$ and $F:=(F^k)_{1\leq k\leq N}$ of $\mathscr{C}^\infty(\Omega)$ functions such that
$b^k \ge 1$ and $F^k \le 0$, and the associated sequence $\Phi:=(\Phi^k)_{1\leq k \leq N}$, with $r:=r_1$. As shown in 
Lemma \ref{lem:dual_classic2}, the functions $\Phi^k$ belong to $\W^{2,p'}(\Omega)$ (for all $p>1$) 
and are hence admissible in the weak formulation \eqref{eq:weak_form}. We may therefore
 write, for all $k\in\llbracket 1, N\rrbracket$, taking $\Phi^{k}$ for test function in the weak formulation on $u_1^k$:
\begin{align*}
\frac{1}{\tau}\langle \Phi^k,u^{k}_1\rangle_{\L^2(\Omega)}-\frac{1}{\tau}\langle \Phi^k,u^{k-1}_1\rangle_{\L^2(\Omega)} &- \left\langle \Delta \Phi^k,\Big[ a_{11}^\ep(u^k_1)+u_1^k a_{12}^\ep (u_2^k) \Big]\right\rangle_{\L^2(\Omega)} \\
 &+\ep \langle \Phi^k- \Delta \Phi^k,w_1^{k}\rangle_{\L^2(\Omega)} = \left\langle \Phi^k,R_{12}^\ep(u_1^{k},u_2^{k}) \right\rangle.
\end{align*}
Thanks to the regularity of $\Phi^k$, one can take $u^k_1$ as a test function in the $k$-th equation \eqref{eq:kth}, so that 
\begin{align*}
\frac{1}{\tau}\langle \Phi^{k+1},u^{k}_1\rangle_{\L^2(\Omega)}-\frac{1}{\tau}\langle \Phi^k,u^{k}_1\rangle_{\L^2(\Omega)} + \langle \Delta \Phi^k,b^k u_1^k\rangle_{\L^2(\Omega)} = \langle \sqrt{b^k}F^k, u^k_1 \rangle_{\L^2(\Omega)}-r_1\langle \Phi^k, u^k_1 \rangle_{\L^2(\Omega)}.
\end{align*}
We hence have, for all $k\in\llbracket 1, N\rrbracket$, adding the two previous equations,
\begin{align*}
\hspace{-1cm}\frac{1}{\tau}\langle \Phi^{k+1},u^{k}_1\rangle_{\L^2(\Omega)}-\frac{1}{\tau}\langle \Phi^k,u^{k-1}_1\rangle_{\L^2(\Omega)} &+ \langle \Delta \Phi^k,\big[b^k-d_{11}^\ep(u_1^k)-a_{12}^\ep(u_2^k)\big]u_1^k\rangle_{\L^2(\Omega)} \\&+ \ep \langle \Phi^k- \Delta \Phi^k,w_1^{k}\rangle_{\L^2(\Omega)}= \langle \sqrt{b^k}F^k, u^k_1 \rangle_{\L^2(\Omega)}- \left\langle \Phi^k,R_{12}^{-,\ep}(u_1^{k},u_2^{k}) \right\rangle,
\end{align*}
recalling that $a_{12}^\ep(x) := a_{12}(x)+\ep x$, $a_{11}^\ep(x) := x d^\ep_{11}(x)$, $d_{11}^\ep := \gamma_\ep(d_{11})$, and the decomposition $R_{12}^\ep = R_{12}^+ - R_{12}^{-,\ep}$.
Since $\Phi^k$ is nonnegative, we get, if we denote $c^k:=b^k-d_{11}^\ep(u_1^k)-a_{12}^\ep(u_2^k)$, after summing over $k\in\llbracket 1,N\rrbracket $  (and recalling that $\Phi^{N+1}=0$),
\begin{align}
\label{ineq:sumtot}-\frac{1}{\tau}\langle \Phi^1,u^{0}_1\rangle_{\L^2(\Omega)} +\sum_{k=1}^N \langle \Delta \Phi^k,c^k u_1^k\rangle_{\L^2(\Omega)} + \sum_{k=1}^N\ep \langle \Phi^k,w_1^{k}\rangle_{\H^1(\Omega)}\leq \sum_{k=1}^N\langle \sqrt{b^k}F^k, u^k_1 \rangle_{\L^2(\Omega)}.
\end{align}
\vspace{2mm}
The following approximation Lemma will help us to handle the sequence ${c:=(c^k)_{1\leq k \leq N}}$:

\begin{lem}\label{lem:approxthomas}
There exists a sequence of families $(b_m)_{m\in\N}:=(b^1_m,b^2_m,\dots,b^N_m)_{m\in\N}$ of $\mathscr{C}^\infty(\Omega)$ functions such that
\begin{small}
\begin{align}
\label{thomas1}\forall (k,m)\in\llbracket 1,N\rrbracket\times \N,\quad b^k_m &\geq 1,\\
\label{thomas2}\forall k\in\llbracket 1,N\rrbracket,\quad\int_{\Omega} |b^k_m-d_{11}^\ep(u_1^k)-a_{12}^\ep(u_2^k)|(1+u_1^k) \,\dd x &\operatorname*{\longrightarrow}_{m\rightarrow +\infty} 0,\\
\label{thomas3}\forall k\in\llbracket 1,N\rrbracket,\quad\int_{\Omega} \left|\sqrt{b^k_m}-\sqrt{d_{11}^\ep(u_1^k)+a_{12}^\ep(u_2^k)}\right|(1+u_1^k) \, \dd x &\operatorname*{\longrightarrow}_{m\rightarrow +\infty} 0,\\
\label{thomas4}\Big\|\sqrt{\underline{b}_{\hspace{1pt} m}^{\boldsymbol{\tau}}}\Big\|_{\L^2(Q_T)}^2 \operatorname*{\longrightarrow}_{m\rightarrow +\infty} \sum_{k=1}^N \tau \int_\Omega |d_{11}^\ep(u_1^k)+a_{12}^\ep(u_2^k) |\, \dd x&=\left\|\sqrt{d_{11}^\ep\big(\underline{u}_1^{\boldsymbol{\tau}}\big)+a_{12}^\ep\big({\underline{u}_{\hspace{1pt} 2}^{\boldsymbol{\tau}}}\big)}\right\|_{\L^2(Q_T)}^2.
\end{align}
\end{small}
\end{lem}
\begin{proof}
This Lemma is a rather direct consequence of the regularity of the borelian measure ${(1+u_1^k)\,\dd x}$: we know 
 that $d_{11}^\ep(u_1^k) + a_{12}^\ep(u_2^k)\in \L^1\big(\Omega,(1+u_1^k)\dd x\big)$ and is lower bounded by a strictly positive constant, we may hence approximate it in this space by smooth $\mathscr{C}^\infty(\Omega)$ functions, which are still lower bounded by a strictly positive constant. This gives \eqref{thomas1}--\eqref{thomas2}. Using the mentioned lower bound, we have
\begin{align*}
\left|\sqrt{b^k_m}-\sqrt{d_{11}^\ep(u_1^k)+a_{12}^\ep(u_2^k)}\right| &\leq \left|\sqrt{b^k_m}-\sqrt{d_{11}^\ep(u_1^k)+a_{12}^\ep(u_2^k)}\right| \Big[\sqrt{b^k_m}+\sqrt{d_{11}^\ep(u_1^k)+a_{12}^\ep(u_2^k)}\Big] \\
&= |b^k_m-d_{11}^\ep(u_1^k)-a_{12}^\ep(u_2^k)|,
\end{align*}
so that we also get \eqref{thomas3}. Noticing that
\begin{align*}
\Big\|\sqrt{{\underline{b}_{\hspace{1pt} m}^{\boldsymbol{\tau}}}}\Big\|_{\L^2(Q_T)}^2 = \big\|{\underline{b}_{\hspace{1pt} m}^{\boldsymbol{\tau}}}\big\|_{\L^1(Q_T)} = \sum_{k=1}^N \tau \int_\Omega |b_m^k| \dd x ,
\end{align*}
we see that \eqref{thomas4} is a simple consequence of the convergence \eqref{thomas2}.
\end{proof}
\vspace{1cm}
We now fix a family $F:=(F^k)_{1\leq k \leq N}$ of $\mathscr{C}^\infty(\Omega)$ nonpositive functions, and for each family $b_m$ of the sequence $(b_m)_{m\in\N}$ defined in Lemma \ref{lem:approxthomas}, we define the corresponding family $\Phi_m:=(\Phi_m^1,\Phi_m^2,\dots,\Phi_m^N)$, using Lemma \ref{lem:dual_classic}. The previous estimate \eqref{ineq:sumtot} can now be written
\begin{align*}
-\langle \Phi^1_m,u^{0}_1\rangle_{\L^2(\Omega)} +\sum_{k=1}^N \tau \langle  \Delta \Phi^k_m,c^k_m u_1^k\rangle_{\L^2(\Omega)} + \sum_{k=1}^N \tau\ep \langle \Phi^k_m,w_1^{k}\rangle_{\H^1(\Omega)}\leq \sum_{k=1}^N \tau \langle \sqrt{b^k_m}F^k, u^k_1 \rangle_{\L^2(\Omega)},
\end{align*}
where $c^k_m:=b^k_m-d_{11}^\ep(u_1^k)-a_{12}^\ep(u_2^k)$. Since the right-hand side is nonpositive, we may write
\begin{align*}
\text{$\fbox{1}$}:= \left|\sum_{k=1}^N \tau\langle \sqrt{b^k_m}F^k, u^k_1 \rangle_{\L^2(\Omega)}\right| & \leq\,  \stackrel{\text{$\fbox{2}$}}{\overbrace{\langle \Phi^1_m,u^{0}_1\rangle_{\L^2(\Omega)}}} +\stackrel{\text{$\fbox{3}$}}{\overbrace{\sum_{k=1}^N \tau \left|\langle \Delta \Phi^k_m,c^k_m u_1^k\rangle_{\L^2(\Omega)}\right|}} \\
&+\stackrel{\text{$\fbox{4}$}}{\overbrace{\sum_{k=1}^N \tau \ep \left| \langle \Phi^k_m,w_1^{k}\rangle_{\H^1(\Omega)}\right|}}.
\end{align*}
Using \eqref{ineq:duality_bound3} of Lemma \ref{lem:dual_classic} and \eqref{thomas4}, we get (for a given $\tau, \var$),
\begin{align*}
\left|\text{$\fbox{2}$}\right| &\leq C_\Omega \|\uu^0\|_{\L^2(\Omega)} \Bigg[e^{r_1(\tau) T}+\Big\|\sqrt{{\underline{b}_{\hspace{1pt} m}^{\boldsymbol{\tau}}}}\Big\|_{\L^2(Q_T)}\Bigg]\big\|\underline{F}^{\boldsymbol{\tau}}\big\|_{\L^2(Q_T)}e^{r_1(\tau) T} \\
&\operatorname*{\longrightarrow}_{m\rightarrow +\infty} C_\Omega \|\uu^0\|_{\L^2(\Omega)} \Bigg[e^{r_1(\tau) T}+\left\|\sqrt{d_{11}^\ep\big({\underline{u}_{\hspace{1pt} 1}^{\boldsymbol{\tau}}}\big)+a_{12}^\ep\big({\underline{u}_{\hspace{1pt} 2}^{\boldsymbol{\tau}}}\big)}\right\|_{\L^2(Q_T)}\Bigg]\big\|\underline{F}^{\boldsymbol{\tau}}\big\|_{\L^2(Q_T)}e^{r_1(\tau) T}.
\end{align*}
Using Lemma \ref{lem:dual_classic2}, we get on the other hand
\begin{align*}
\left|\text{$\fbox{3}$}\right| &\leq \tau N  \sup_{1\leq k \leq N} \Big\{\|\Delta \Phi_m^k\|_{\L^\infty(\Omega)} \|c_m^k u_1^k\|_{\L^1(\Omega)}\Big\}\\
&\leq    \tau N C_{\tau,\Omega}\Bigg[1+\Big\|\sqrt{{\underline{b}_{\hspace{1pt} m}^{\boldsymbol{\tau}}}}\Big\|_{\L^2(Q_T)}\Bigg]\big\|\underline{F}^{\boldsymbol{\tau}}\big\|_{\L^\infty(Q_T)}\sup_{1\leq k \leq N} \|c_m^k u_1^k\|_{\L^1(\Omega)}
\operatorname*{\longrightarrow}_{m\rightarrow +\infty} 0,
\end{align*}
using \eqref{thomas2} and \eqref{thomas4} for the convergence. In order to treat $\fbox{4}$, notice that we know, from inequality \eqref{ineq:afterlim1}:
\begin{align*}
\ep \tau \sum_{k=1}^N\|w_1^k\|_{\H^1(\Omega)}^2 \leq K_T (\mathscr{E}_\ep(\uu^0)+1),
\end{align*}
hence, using again \eqref{ineq:duality_bound3} of Lemma \ref{lem:dual_classic}, we get ($\tau N=T$)
\begin{align*}
\left|\text{$\fbox{4}$}\right| &\leq e^{r_1(\tau) T} C_\Omega \Bigg[e^{r_1(\tau) T}+\Big\|\sqrt{{\underline{b}_{\hspace{1pt} m}^{\boldsymbol{\tau}}}}\Big\|_{\L^2(Q_T)}\Bigg]\big\|\underline{F}^{\boldsymbol{\tau}}\big\|_{\L^\infty(Q_T)}  \sum_{k=1}^N \tau \ep \|w_1^k\|_{\H^1(\Omega)}\\
&\leq e^{r_1(\tau) T} C_\Omega \Bigg[e^{r_1(\tau) T}+\Big\|\sqrt{{\underline{b}_{\hspace{1pt} m}^{\boldsymbol{\tau}}}}\Big\|_{\L^2(Q_T)}\Bigg]\big\|\underline{F}^{\boldsymbol{\tau}}\big\|_{\L^\infty(Q_T)}\tau \ep \sqrt{N} \Bigg\{ \sum_{k=1}^N  \|w_1^k\|^2_{\H^1(\Omega)}\Bigg\}^{1/2}\\
&\leq e^{r_1(\tau) T} C_\Omega \Bigg[e^{r_1(\tau) T}+\Big\|\sqrt{{\underline{b}_{\hspace{1pt} m}^{\boldsymbol{\tau}}}}\Big\|_{\L^2(Q_T)}\Bigg]\big\|\underline{F}^{\boldsymbol{\tau}}\big\|_{\L^\infty(Q_T)}\sqrt{ \ep T} \sqrt{K_T (\mathscr{E}_\ep(\uu^0)+1)},
\end{align*}
and we may use again \eqref{thomas4} to get that the last upper bound is converging, as $m$ goes to $+\infty$, to
\begin{align*}
e^{r_1(\tau) T} C_\Omega \Bigg[e^{r_1(\tau) T}+\left\|\sqrt{d_{11}^\ep\big({\underline{u}_{\hspace{1pt} 1}^{\boldsymbol{\tau}}}\big)+a_{12}^\ep\big({\underline{u}_{\hspace{1pt} 2}^{\boldsymbol{\tau}}}\big)}\right\|_{\L^2(Q_T)}\Bigg]\big\|\underline{F}^{\boldsymbol{\tau}}\big\|_{\L^2(Q_T)} \sqrt{ \ep T} \sqrt{K_T (\mathscr{E}_\ep(\uu^0)+1)}.
\end{align*}
Finally, because of \eqref{thomas3}, we can see that $\fbox{1}$ converges, as $m$ goes to $+\infty$, to
\begin{align*}
\left|\sum_{k=1}^N \tau \int_{\Omega} u_1^k\sqrt{d_{11}^\ep\big(u_1^k\big)+a_{12}^\ep(u_2^k)} F^k\dd x \right|= \left|\int_{Q_T} {\underline{u}_{\hspace{1pt} 1}^{\boldsymbol{\tau}}}\sqrt{d_{11}^\ep\big({\underline{u}_{\hspace{1pt} 1}^{\boldsymbol{\tau}}}\big)+a_{12}^\ep\big({\underline{u}_{\hspace{1pt} 2}^{\boldsymbol{\tau}}}\big)}  \underline{F}^{\boldsymbol{\tau}} \dd x\,\dd t\right|.
\end{align*}
All the previous estimates give hence, denoting $\underline{h}^{\boldsymbol{\tau}} := {\underline{u}_{\hspace{1pt} 1}^{\boldsymbol{\tau}}}\sqrt{d_{11}^\ep\big({\underline{u}_{\hspace{1pt} 1}^{\boldsymbol{\tau}}}\big)+a_{12}^\ep\big({\underline{u}_{\hspace{1pt} 2}^{\boldsymbol{\tau}}}\big)}$,
\begin{small}
\begin{align*}
\hspace{-1.5cm}\left| \int_{Q_T} \underline{h}^{\boldsymbol{\tau}} \underline{F}^{\boldsymbol{\tau}} \dd x\,\dd t \right| \leq e^{r_1(\tau) T} C_\Omega \Bigg[e^{r_1(\tau) T}+\left\|\sqrt{d_{11}^\ep\big({\underline{u}_{\hspace{1pt} 1}^{\boldsymbol{\tau}}}\big)+a_{12}^\ep\big({\underline{u}_{\hspace{1pt} 2}^{\boldsymbol{\tau}}}\big)}\right\|_{\L^2(Q_T)}\Bigg]
\end{align*}
\begin{align*}
 \times\, \Big[\sqrt{ \ep T} \sqrt{K_T (\mathscr{E}_\ep(\uu^0)+1)} + \|\uu^0\|_{\L^2(\Omega)}\Big]\big\|\underline{F}^{\boldsymbol{\tau}}\big\|_{\L^2(Q_T)}.
\end{align*}
\end{small}
Since $\underline{h}^{\boldsymbol{\tau}}$ is a step (in time) nonnegative function and the previous holds true for all non-positive smooth (in $x$) step (in time) functions, we have then by duality
\begin{small}
\begin{align*}
\hspace{-1.5cm}\left\|{\underline{u}_{\hspace{1pt} 1}^{\boldsymbol{\tau}}}\sqrt{d_{11}^\ep\big({\underline{u}_{\hspace{1pt} 1}^{\boldsymbol{\tau}}}\big)+a_{12}^\ep\big({\underline{u}_{\hspace{1pt} 2}^{\boldsymbol{\tau}}}\big)}\right\|_{\L^2(Q_T)} \leq e^{2r_1(\tau) T} C_\Omega \Bigg[1+\left\|\sqrt{d_{11}^\ep\big({\underline{u}_{\hspace{1pt} 1}^{\boldsymbol{\tau}}}\big)+a_{12}^\ep\big({\underline{u}_{\hspace{1pt} 2}^{\boldsymbol{\tau}}}\big)}\right\|_{\L^2(Q_T)}\Bigg]
\end{align*}
\begin{align*}
\times\, \Big[\sqrt{ \ep T} \sqrt{K_T (\mathscr{E}_\ep(\uu^0)+1)} + \|\uu^0\|_{\L^2(\Omega)}\Big].
\end{align*}
\end{small}
Hence, in $\L^1$ norms, we get
\begin{small}
\begin{align*}
\hspace{-1.5cm}\left\|{\underline{u}_{\hspace{1pt} 1}^{\boldsymbol{\tau}}}^2 \Big[d_{11}^\ep\big({\underline{u}_{\hspace{1pt} 1}^{\boldsymbol{\tau}}}\big)+a_{12}^\ep({\underline{u}_{\hspace{1pt} 2}^{\boldsymbol{\tau}}}\big)\Big]\right\|_{\L^1(Q_T)} \leq e^{4r_1(\tau) T} C_\Omega \Bigg[1+\left\|d_{11}^\ep\big({\underline{u}_{\hspace{1pt} 1}^{\boldsymbol{\tau}}}\big)+a_{12}^\ep\big({\underline{u}_{\hspace{1pt} 2}^{\boldsymbol{\tau}}}\big)\right\|_{\L^1(Q_T)}\Bigg]
\end{align*}
\begin{align*}
\times\,
 \Big[ \ep T K_T (\mathscr{E}_\ep(\uu^0)+1) + \|\uu^0\|_{\L^2(\Omega)}^2\Big].
\end{align*}
\end{small}
Since $\ep < 1$ and $\displaystyle (r_1(\tau))_\tau \operatorname*{\longrightarrow}_{\tau \rightarrow 0} r_1$,  the previous inequality can be written in a simpler form :
\begin{small}
\begin{align*}
\hspace{-1.5cm}\left\|{\underline{u}_{\hspace{1pt} 1}^{\boldsymbol{\tau}}}^2 \Big[d_{11}^\ep\big({\underline{u}_{\hspace{1pt} 1}^{\boldsymbol{\tau}}}\big)+a_{12}^\ep({\underline{u}_{\hspace{1pt} 2}^{\boldsymbol{\tau}}}\big)\Big]\right\|_{\L^1(Q_T)} \leq D \Bigg[1+\left\|d_{11}^\ep\big({\underline{u}_{\hspace{1pt} 1}^{\boldsymbol{\tau}}}\big)\right\|_{\L^1(\Omega)}+\left\|a_{12}^\ep\big({\underline{u}_{\hspace{1pt} 2}^{\boldsymbol{\tau}}}\big)\right\|_{\L^1(Q_T)}\Bigg],
\end{align*}
\end{small}
where the constant $D$ only depends on $T,\Omega$ and $\uu^0$.
Because of the concavity of $a_{12}^\ep$ (and since it vanishes in $0$), we have, with some constant $C_1, C_2$ \textbf{not depending} on $\ep$, 
\begin{align*}
a_{12}^\ep(x) \leq C_1+ C_2\,\psi_2^\ep(x),
\end{align*}
so that, from \eqref{ineq:afterlim1}, we have
\begin{small}
\begin{align*}
\hspace{-1.5cm}\left\|{\underline{u}_{\hspace{1pt} 1}^{\boldsymbol{\tau}}}^2 \Big[d_{11}^\ep\big({\underline{u}_{\hspace{1pt} 1}^{\boldsymbol{\tau}}}\big)+a_{12}^\ep({\underline{u}_{\hspace{1pt} 2}^{\boldsymbol{\tau}}}\big)\Big]\right\|_{\L^1(Q_T)} \leq D \Bigg[1+\left\|d_{11}^\ep\big({\underline{u}_{\hspace{1pt} 1}^{\boldsymbol{\tau}}}\big)\right\|_{\L^1(\Omega)}+K_T (\mathscr{E}_\ep(\uu^0)+1)\Bigg],
\end{align*}
\end{small}
that we may write (changing the definition of $D$)
\begin{small}
\begin{align*}
\hspace{-1.5cm}\left\|{\underline{u}_{\hspace{1pt} 1}^{\boldsymbol{\tau}}}^2 \Big[d_{11}^\ep\big({\underline{u}_{\hspace{1pt} 1}^{\boldsymbol{\tau}}}\big)+a_{12}^\ep({\underline{u}_{\hspace{1pt} 2}^{\boldsymbol{\tau}}}\big)\Big]\right\|_{\L^1(Q_T)} \leq D \Bigg[1+\left\|d_{11}^\ep\big({\underline{u}_{\hspace{1pt} 1}^{\boldsymbol{\tau}}}\big)\right\|_{\L^1(\Omega)}\Bigg].
\end{align*}
\end{small}
But we have then, since $d_{11}^\ep \leq d_{11}$ which is a non-decreasing function,
\begin{small}
\begin{align*}
\hspace{-1.5cm}\left\|{\underline{u}_{\hspace{1pt} 1}^{\boldsymbol{\tau}}}^2 \Big[d_{11}^\ep\big({\underline{u}_{\hspace{1pt} 1}^{\boldsymbol{\tau}}}\big)+a_{12}^\ep({\underline{u}_{\hspace{1pt} 2}^{\boldsymbol{\tau}}}\big)\Big]\right\|_{\L^1(Q_T)} &\leq D \Bigg[1+\left\|\mathbf{1}_{{\underline{u}_{\hspace{1pt} 1}^{\boldsymbol{\tau}}}\geq \sqrt{2D}}\,d_{11}^\ep\big({\underline{u}_{\hspace{1pt} 1}^{\boldsymbol{\tau}}}\big)\right\|_{\L^1(\Omega)} +  \left\|\mathbf{1}_{{\underline{u}_{\hspace{1pt} 1}^{\boldsymbol{\tau}}}< \sqrt{2D}}\,d_{11}^\ep\big({\underline{u}_{\hspace{1pt} 1}^{\boldsymbol{\tau}}}\big)\right\|_{\L^1(\Omega)}\Bigg] \\
&\leq D \Bigg[1+\frac{1}{2D}\left\| {\underline{u}_{\hspace{1pt} 1}^{\boldsymbol{\tau}}}^2d_{11}^\ep\big({\underline{u}_{\hspace{1pt} 1}^{\boldsymbol{\tau}}}\big)\right\|_{\L^1(\Omega)} + d_{11}(\sqrt{2D})\mu(\Omega)\Bigg],
\end{align*}
\end{small}
so that
\begin{align*}
\hspace{-1.5cm}\left\|{\underline{u}_{\hspace{1pt} 1}^{\boldsymbol{\tau}}}^2 \Big[\frac{1}{2}d_{11}^\ep\big({\underline{u}_{\hspace{1pt} 1}^{\boldsymbol{\tau}}}\big)+a_{12}^\ep({\underline{u}_{\hspace{1pt} 2}^{\boldsymbol{\tau}}}\big)\Big]\right\|_{\L^1(Q_T)} 
&\leq D \Big[1+ d_{11}(\sqrt{2D})\mu(\Omega)\Big],
\end{align*}
that is to say, eq. (\ref{ineq:dual}).
\end{proof}

\section{Proof of the main Theorem} \label{section6}

This section is dedicated to the proof of our main Theorem.
\bigskip

\begin{proof}
Thanks to Proposition \ref{propo:exist2}, we can consider a solution
 ${\underline{u}_{\hspace{1pt} 1}^{\boldsymbol{\tau}}}^{\ep}, {\underline{u}_{\hspace{1pt} 2}^{\boldsymbol{\tau}}}^{\ep}$ to system \eqref{eq:weak_form}.
 For the sake of clarity, we shall denote this solution by ${\underline{u}_{\hspace{1pt} 1}^{\boldsymbol{\tau}}}, {\underline{u}_{\hspace{1pt} 2}^{\boldsymbol{\tau}}}$.
 Let us keep in mind that since we are dealing with the limit $(\tau,\ep)\rightarrow(0,0)$, we shall have to use bounds that are uniform w.r.t. these two parameters. 
 Hence in the sequel, if not mentioned, the term ``bounded'' will always have to be understood as ``uniformly w.r.t. $\tau$ and $\ep$''. We shall only work with  ${\underline{u}_{\hspace{1pt} 1}^{\boldsymbol{\tau}}}$, the study of  ${\underline{u}_{\hspace{1pt} 2}^{\boldsymbol{\tau}}}$ being exactly identical.
\vspace{2mm}\\
We start with a simple computation. Take some real number $\eta>0$ and write
\begin{align*}
{\underline{u}_{\hspace{1pt} 1}^{\boldsymbol{\tau}}}(t+\eta,x)-{\underline{u}_{\hspace{1pt} 1}^{\boldsymbol{\tau}}}(t,x)& = \sum_{k=1}^N \big[\mathbf{1}_{](k-1)\tau,k\tau]}(t+\eta)-\mathbf{1}_{](k-1)\tau,k\tau]}(t)\big]u_1^k(x)\\
&= \sum_{k=1}^N \big[\mathbf{1}_{](k-1)\tau-\eta,k\tau-\eta]}(t)-\mathbf{1}_{](k-1),k\tau]}(t)\big]u_1^k(x).
\end{align*}
Hence, using
\begin{align*}
\mathbf{1}_{](k-1)\tau-\eta,k\tau-\eta]}(t)-\mathbf{1}_{](k-1)\tau,k\tau]}(t) = \mathbf{1}_{](k-1)\tau-\eta,(k-1)\tau]}(t)-\mathbf{1}_{]k\tau-\eta,k\tau]}(t),
\end{align*}
we get
\begin{align*}
{\underline{u}_{\hspace{1pt} 1}^{\boldsymbol{\tau}}}(t+\eta,x)-{\underline{u}_{\hspace{1pt} 1}^{\boldsymbol{\tau}}}(t,x) &= \sum_{k=0}^{N-1} \mathbf{1}_{]k\tau-\eta,k\tau]}(t)u_1^{k+1}(x)-\sum_{k=1}^N\mathbf{1}_{]k\tau-\eta,k\tau]}(t)u_1^k(x)\\
&= \sum_{k=1}^{N-1} \mathbf{1}_{]k\tau-\eta,k\tau]}(t)[u_1^{k+1}(x)-u^k_1(x)] + \mathbf{1}_{]-\eta,0]}(t)u_1^1(x) - \mathbf{1}_{]T-\eta,T]}(t)u_1^N(x).
\end{align*}
 If we denote by $\sigma_\eta$ the translation operator $g(\cdot,x) \mapsto g(\cdot + \eta,x)$, we have, introducing the space $\E_\eta := \L^1([0,T-\eta];\W^{2,\infty}(\Omega)')$, 
\begin{align*}
\| \sigma_\eta {\underline{u}_{\hspace{1pt} 1}^{\boldsymbol{\tau}}} -  {\underline{u}_{\hspace{1pt} 1}^{\boldsymbol{\tau}}} \|_{\E_\eta} \leq \eta\sum_{k=1}^{N-1}  \|u_1^{k+1}-u_1^k\|_{\W^{2,\infty}(\Omega)'}.
\end{align*} 
Now recall that, for all $j\in\llbracket 1, N\rrbracket$,
\begin{align*}
\partial_\tau u_1^j - \Delta\Big([d_{11}^\ep(u_1^j)+a_{12}^\ep(u_2^j)]{u_1^j}\Big) + \ep (w_1^j-\Delta w_1^j)= R_{12}^\ep(\uu^{j}),
\end{align*}
weakly. We then have (taking $j=k+1$), since $\H^1(\Omega) \hookrightarrow \H^{-1}(\Omega) \hookrightarrow \W^{2,\infty}(\Omega)'$, and $\L^1(\Omega)\hookrightarrow \W^{2,\infty}(\Omega)'$ 
\begin{align*}
\hspace{-1cm}\frac{1}{\tau}\|u_1^{k+1} - u_1^{k}\|_{\W^{2,\infty}(\Omega)'} &\leq \left\|[d_{11}^\ep(u_
1^{k+1})+a_{12}^\ep(u_2^{k+1})]u_1^{k+1}\right\|_{\L^1(\Omega)} + \ep C_\Omega \|w_1^{k+1}\|_{\H^1(\Omega)} + \|R^\ep_{12}(\uu^{k+1})\|_{\L^1(\Omega)} \\
&= \left\|\big[d_{11}^\ep(u_1^{k+1})+a_{12}^\ep(u_2^{k+1})\big]u_1^{k+1} \mathbf{1}_{u_1^{k+1}\geq 1}\right\|_{\L^1(\Omega)}  \\
&+ \left\|\big[d_{11}^\ep(u_1^{k+1})+a_{12}^\ep(u_2^{k+1})\big]u_1^{k+1} \mathbf{1}_{u_1^{k+1}< 1}\right\|_{\L^1(\Omega)}\\
&+\ep C_\Omega \|w_1^{k+1}\|_{\H^1(\Omega)} 
+ \|R^{\ep,-}_{12}(\uu^{k+1})\|_{\L^1(\Omega)} + r_1 \| u_1^{k+1}\|_{\L^1(\Omega)}.
\end{align*}
Using $\mathbf{1}_\Omega=\mathbf{1}_{u_1^{k+1}\geq 1} + \mathbf{1}_{u_1^{k+1}<1}$, we have 
\begin{align*}
\hspace{-1.7cm}\left\|[d_{11}^\ep(u_
1^{k+1})+a_{12}^\ep(u_2^{k+1})]u_1^{k+1}\right\|_{\L^1(\Omega)} \leq \left\|u_1^{k+1}\sqrt{d_{11}^\ep(u_
1^{k+1})+a_{12}^\ep(u_2^{k+1})}\right\|_{\L^2(\Omega)}^2 
\end{align*}
\begin{align*}
\, + \mu(\Omega)\sup_{x\in[0,1]} d_{11}^\ep(x)  + \|a_{12}^\ep(u_2^{k+1})\|_{\L^1(\Omega)}.
\end{align*}
Since for $\ep$ small enough, $d_{11}^\ep$ coincides with $d_{11}$ on $[0,1]$, using also the last point of Lemma \ref{lem:ineqconv:approx} to get $a_{12}^\ep \leq D (2 + \psi_2^\ep)$, we have eventually (recalling \eqref{ineq:afterlim1}) for some constant $C$ not depending on $\ep$,
\begin{align*}
\left\|[d_{11}^\ep(u_
1^{k+1})+a_{12}^\ep(u_2^{k+1})]u_1^{k+1}\right\|_{\L^1(\Omega)} &\leq \left\|u_1^{k+1}\sqrt{d_{11}^\ep(u_
1^{k+1})+a_{12}^\ep(u_2^{k+1})}\right\|_{\L^2(\Omega)}^2 +  C K_T(\mathscr{E}_\ep(\uu^0) +1)\\
&\leq \left\|u_1^{k+1}\sqrt{d_{11}^\ep(u_1^{k+1})+a_{12}^\ep(u_2^{k+1})}\right\|_{\L^2(\Omega)}^2 +  C,
\end{align*}
where we change the definition of $C$, so that finally
\begin{align*}
\hspace{-1cm}\big\| \sigma_\eta {\underline{u}_{\hspace{1pt} 1}^{\boldsymbol{\tau}}} -  {\underline{u}_{\hspace{1pt} 1}^{\boldsymbol{\tau}}} \big\|_{\E_\eta} &\leq \eta \sum_{k=1}^{N-1} \tau\left\{ \left\|u_1^{k+1}\sqrt{d_{11}^\ep(u_1^{k+1})+a_{12}^\ep(u_2^{k+1})}\right\|_{\L^2(\Omega)}^2  + C+\ep C_\Omega \|w_1^{k+1}\|_{\H^1(\Omega)} \right\} \\
&+ \eta \sum_{k=1}^{N-1} \tau\left\{ \|R^{\ep,-}_{12}(\uu^{k+1})\|_{\L^1(\Omega)} + r_1 \| u_1^{k+1}\|_{\L^1(\Omega)} \right\} \\ 
&\leq 4 \eta \Big\|{\underline{u}_{\hspace{1pt} 1}^{\boldsymbol{\tau}}}\sqrt{d_{11}^\ep\big({\underline{u}_{\hspace{1pt} 1}^{\boldsymbol{\tau}}}\big)+a_{12}^\ep\big({\underline{u}_{\hspace{1pt} 2}^{\boldsymbol{\tau}}}\big)}\Big\|_{\L^2(Q_T)}^2 + \eta TC + \eta \ep \tau C_\Omega \sum_{k=1}^N\|\ww^k\|_{\H^1(\Omega)}\\
&+ \eta \sum_{k=1}^{N-1} \tau\left\{ \|R^{\ep,-}_{12}(\uu^{k+1})\|_{\L^1(\Omega)} + r_1 \| u_1^{k+1}\|_{\L^1(\Omega)} \right\},
\end{align*}
which, using \eqref{ineq:dual}, \eqref{ineq:afterlim1}, \eqref{ineq:afterlim3} and \eqref{ineq:afterlim4},
 is going to $0$ as $\eta\rightarrow 0$, uniformly w.r.t. $\tau$ and $\ep$.
\vspace{2mm}\\
We also have 
\begin{align*}
|\nabla {\underline{u}_{\hspace{1pt} 1}^{\boldsymbol{\tau}}}|(t,x) = \sum_{k=1}^{N} \mathbf{1}_{](k-1)\tau,k\tau]}(t) |\nabla u^k_1|(x).
\end{align*}
Recall the notation $\beta_\alpha(x):=x^{\frac{1-\alpha}{2}}$. Given $v:\Omega\rightarrow \R_+$ such that $\beta_\alpha(v)\in\H^1(\Omega)$ (and considering for example the sequence $(v+1/m)_{m\in\N}$), we get 
\begin{align*}
\frac{1-\alpha}{2}v^{-\frac{\alpha+1}{2}}\nabla v = \nabla \beta_\alpha(v)
\end{align*}
weakly. Hence, using \eqref{ineq:afterlim2}, we have for all $k\in\llbracket 1,N \rrbracket$: 
\begin{align*}
\tau \sum_{j=1}^k \int_\Omega |u_1^j|^{-\alpha-1}|\nabla u_1^{j}|^2 \dd x \leq  K_T(\mathscr{E}_\ep(\uu^0) +1).
\end{align*}
We know that
\begin{align*}
|\nabla {\underline{u}_{\hspace{1pt} 1}^{\boldsymbol{\tau}}}|(t,x) = \sum_{k=1}^{N} \mathbf{1}_{](k-1)\tau,k\tau]}(t) |u_1^k|^{\frac{\alpha+1}{2}} |u_1^k|^{-\frac{\alpha+1}{2}}|\nabla u_1^{k}|.
\end{align*}
Hence, using the dual estimate \eqref{ineq:dual}, we see that $\nabla {\underline{u}_{\hspace{1pt} 1}^{\boldsymbol{\tau}}}$ is bounded in $\L^{q_\alpha}(Q_T)$,
 where $q_\alpha$ is defined by the equality ($\alpha\in]0,1[$)
\begin{align*}
\frac{\alpha+1}{4} + \frac{1}{2} = \frac{1}{q_\alpha}, 
\end{align*}
that is $q_\alpha=8/(2\alpha+6)>1$. At this stage we obtained, denoting $\E_\eta:=\L^1\big([0,T-\eta];\W^{2,\infty}(\Omega)'\big)$,
\begin{itemize}
\item[$\bullet$] $\| \sigma_\eta {\underline{u}_{\hspace{1pt} 1}^{\boldsymbol{\tau}}} -  {\underline{u}_{\hspace{1pt} 1}^{\boldsymbol{\tau}}} \|_{\E_\eta}$ goes to $0$ with $\eta$ uniformly in $\tau$ and $\ep$,
\item[$\bullet$]$({\underline{u}_{\hspace{1pt} 1}^{\boldsymbol{\tau}}})_{\tau}$ is bounded in $\L^{q_\alpha}\big([0,T];\W^{1,{q_\alpha}}(\Omega)\big)$.
\end{itemize}
We know that $\W^{1,q_\alpha}(\Omega)\hookrightarrow \L^r(\Omega) \hookrightarrow \W^{2,\infty}(\Omega)'$, where the first injection is compact, the second one is continuous, and $r$ denotes any real number of the interval $[1,q^\star_\alpha[$. We may hence use Theorem $5$ of \cite{simon} to get the compactness of $({\underline{u}_{\hspace{1pt} 1}^{\boldsymbol{\tau}}})_{\tau}$ in $\L^{1}\big([0,T];\L^r(\Omega)\big)$.\vspace{2mm}\\
From the definition of ${\underline{u}_{\hspace{1pt} 1}^{\boldsymbol{\tau}}}$:
\begin{align*}
{\underline{u}_{\hspace{1pt} 1}^{\boldsymbol{\tau}}}(t,x):= \sum_{k=1}^N u_1^k(x) \mathbf{1}_{](k-1)\tau,k\tau]} (t),
\end{align*}
we get
\begin{align*}
\sigma_{-\tau}{\underline{u}_{\hspace{1pt} 1}^{\boldsymbol{\tau}}}(t,x):= {\underline{u}_{\hspace{1pt} 1}^{\boldsymbol{\tau}}}(t-\tau,x)&= \sum_{k=1}^N u_1^k(x) \mathbf{1}_{](k-1)\tau,k\tau]} (t-\tau)\\
&= \sum_{k=1}^N u_1^k(x) \mathbf{1}_{]k\tau,(k+1)\tau]} (t)\\
&= \sum_{k=2}^{N+1} u_1^{k-1}(x) \mathbf{1}_{](k-1)\tau,k\tau]} (t).
\end{align*}
We hence obtain the following expression for the rate of growth
\begin{align*}
\frac{{\underline{u}_{\hspace{1pt} 1}^{\boldsymbol{\tau}}}-\sigma_{-\tau}{\underline{u}_{\hspace{1pt} 1}^{\boldsymbol{\tau}}}}{\tau} (t,x)= \sum_{k=1}^N \mathbf{1}_{](k-1)\tau,k\tau]}(t) \partial_\tau u_1^k(x) + \frac{1}{\tau}\mathbf{1}_{]0,\tau]}(t) u^1_1(x) -  \frac{1}{\tau}\mathbf{1}_{]T,T+\tau]}(t) u^N_1(x),
\end{align*}
recalling the definition of the finite difference operator,
\begin{align*}
\partial_\tau u^k_1 := \frac{u^k_1-u_1^{k-1}}{\tau}.
\end{align*}
Now let us fix a test function $\theta \in{\mathscr{D}}([0, T[; \mathscr{C}_\nu^{\infty}(\overline{\Omega}))$.
 We have
\begin{align*}
\int_{0}^T \int_\Omega \frac{{\underline{u}_{\hspace{1pt} 1}^{\boldsymbol{\tau}}}-\sigma_{-\tau}{\underline{u}_{\hspace{1pt} 1}^{\boldsymbol{\tau}}}}{\tau} \cdot \theta (t,x)  \,\dd x\,\dd t = \sum_{k=1}^N \int_{(k-1)\tau}^{k\tau} \langle\partial_\tau u_1^k,  \theta(t)  \rangle\, \dd t + \frac{1}{\tau}\int_0^\tau \langle u^1_1, \theta(t)\rangle\,\dd t,
\end{align*}
where we denote by $\langle \cdot , \cdot\rangle$ the duality bracket $\mathscr{D'}(\Omega)/\mathscr{D}(\Omega)$ (which is simply the integration on $\Omega$ here ...). A change of variable leads to
\begin{align*}
\int_{0}^T \int_\Omega \frac{{\underline{u}_{\hspace{1pt} 1}^{\boldsymbol{\tau}}}-\sigma_{-\tau}{\underline{u}_{\hspace{1pt} 1}^{\boldsymbol{\tau}}}}{\tau} \cdot \theta (t,x)  \,\dd x\,\dd t &= \int_{0}^T \int_\Omega {\underline{u}_{\hspace{1pt} 1}^{\boldsymbol{\tau}}} \cdot \frac{\theta}{\tau} (t,x)  \,\dd x\,\dd t - \int_{-\tau}^{T-\tau} \int_\Omega {\underline{u}_{\hspace{1pt} 1}^{\boldsymbol{\tau}}} \cdot \frac{\sigma_{\tau}\theta}{\tau} (t,x)  \,\dd x\,\dd t\\
&= \int_0^{T-\tau}  \int_\Omega {\underline{u}_{\hspace{1pt} 1}^{\boldsymbol{\tau}}} \cdot \frac{\theta-\sigma_\tau \theta}{\tau} (t,x) \,\dd x\,\dd t \\
&+ \stackrel{\fbox{1}}{\overbrace{\int_{T-\tau}^T \int_\Omega {\underline{u}_{\hspace{1pt} 1}^{\boldsymbol{\tau}}} \cdot \frac{\theta}{\tau} (t,x)  \,\dd x\,\dd t}} - \stackrel{\fbox{2}}{\overbrace{\int_{-\tau}^{0} \int_\Omega {\underline{u}_{\hspace{1pt} 1}^{\boldsymbol{\tau}}} \cdot \frac{\sigma_{\tau}\theta}{\tau} (t,x)  \,\dd x\,\dd t}}.
\end{align*}
Then, $\fbox{1}$ equals to $0$ for $\tau$ small enough because $\theta\in\mathscr{D}([0,T[\times\Omega)$, and
$\fbox{2}$ equals to $0$ for all $\tau$ because of the definition of ${\underline{u}_{\hspace{1pt} 1}^{\boldsymbol{\tau}}}$, so that finally
\begin{align}\label{eq:taux}
 \sum_{k=1}^N \int_{(k-1)\tau}^{k\tau} \langle\partial_\tau u_1^k,  \theta(t)  \rangle\, \dd t + \frac{1}{\tau}\int_0^\tau \langle u^1_1, \theta(t)\rangle\,\dd t &= \int_0^{T-\tau}  \int_\Omega {\underline{u}_{\hspace{1pt} 1}^{\boldsymbol{\tau}}} \cdot \frac{\theta-\sigma_\tau \theta}{\tau} (t,x) \,\dd x\,\dd t.
\end{align}
Since for all $k\in\llbracket 1, N\rrbracket$ we have (in the weak sense)
\begin{align*}
\partial_\tau u_1^k - \Delta\Big(\big[d_{11}^\ep(u_1^k)+a_{12}^\ep(u_2^k)\big]{u_1^k}\Big) + \ep ({w_1}^k-\Delta w_1^k)= R_{12}^\ep(\uu^{k}),
\end{align*}
we also can write, for all $t\in[0,T[$,
\begin{align*}
\langle \partial_\tau u_1^k,\theta(t)\rangle - \Big\langle a_{11}^\ep(u_1^k)+u_1^k a_{12}(u_2^k)+\ep u_1^k u_2^k,\Delta \theta(t)\Big\rangle  + \ep \langle w_1^k, \theta(t)-\Delta \theta(t)\rangle = \langle R_{12}^\ep(\uu^{k}),\theta(t)\rangle,
\end{align*}
so that integrating on $](k-1)\tau,\tau]$ and summing over $k\in\llbracket 1, N\rrbracket$, we get 
\begin{align*}
\sum_{k=1}^N \int_{(k-1)\tau}^{k\tau} \langle\partial_\tau u_1^k,  \theta(t)  \rangle\, \dd t &- \int_0^T \int_\Omega \Big[ a_{11}^\ep\big({\underline{u}_{\hspace{1pt} 1}^{\boldsymbol{\tau}}}\big)+{\underline{u}_{\hspace{1pt} 1}^{\boldsymbol{\tau}}}a_{12}\big({\underline{u}_{\hspace{1pt} 2}^{\boldsymbol{\tau}}}\big)+\ep {\underline{u}_{\hspace{1pt} 1}^{\boldsymbol{\tau}}}{\underline{u}_{\hspace{1pt} 2}^{\boldsymbol{\tau}}} \Big]\cdot \Delta \theta (t,x) \, \dd x\,\dd t \\
&+ \ep \sum_{k=1}^N \int_{(k-1)\tau}^{k\tau}\langle w_1^k,\theta(t)-\Delta\theta(t)\rangle\,\dd t = \int_0^T \int_\Omega R^\ep_{12}(\underline{\uu}^{\boldsymbol{\tau}})(t,x) \theta(t,x)\dd x\,\dd t.
\end{align*}
Using \eqref{eq:taux}, we eventually get
\begin{align*}
 \int_0^{T}  \int_\Omega \mathbf{1}_{[0,T-\tau[}(t){\underline{u}_{\hspace{1pt} 1}^{\boldsymbol{\tau}}} \cdot \frac{\theta-\sigma_\tau \theta}{\tau} (t,x) \,\dd x\,\dd t &- \int_0^T \int_\Omega \Big[a_{11}^\ep\big({\underline{u}_{\hspace{1pt} 1}^{\boldsymbol{\tau}}}\big)+{\underline{u}_{\hspace{1pt} 1}^{\boldsymbol{\tau}}}a_{12}\big({\underline{u}_{\hspace{1pt} 2}^{\boldsymbol{\tau}}}\big)+\ep {\underline{u}_{\hspace{1pt} 1}^{\boldsymbol{\tau}}}{\underline{u}_{\hspace{1pt} 2}^{\boldsymbol{\tau}}} \Big]\cdot \Delta \theta (t,x) \, \dd x\,\dd t \\
&+ \ep \sum_{k=1}^N \int_{(k-1)\tau}^{k\tau}\langle w_1^k,\theta(t)-\Delta\theta(t)\rangle\,\dd t\\
&= \frac{1}{\tau}\int_0^\tau \langle u^1_1, \theta(t)\rangle\,\dd t +  \int_0^T \int_\Omega R^\ep_{12}(\underline{\uu}^{\boldsymbol{\tau}})(t,x)\, \theta(t,x)\,\dd x\,\dd t.
\end{align*}
We now can study the limit $(\tau,\ep)\rightarrow (0,0)$. Thanks to the compactness result that we proved above, we get the existence of $\uu:=(u_1,u_2)\in\L^1\big([0,T];\L^r(\Omega)\big)$ (with $r<q^\star_\alpha$) such that, up to a subsequence, we have
\begin{align*}
({\underline{u}_{\hspace{1pt} 1}^{\boldsymbol{\tau}}})_{\tau,\ep} & \operatorname*{\longrightarrow}_{(\tau,\ep)\rightarrow (0,0)} u_1, \\
({\underline{u}_{\hspace{1pt} 2}^{\boldsymbol{\tau}}})_{\tau,\ep} & \operatorname*{\longrightarrow}_{(\tau,\ep)\rightarrow (0,0)} u_2,
\end{align*}
in $\L^1\big([0,T];\L^r(\Omega)\big)$, and also almost everywhere on $Q_T$. Because of the dual estimate \eqref{ineq:dual} $\Big({\underline{u}_{\hspace{1pt} 1}^{\boldsymbol{\tau}}} \sqrt{a_{12}\big({\underline{u}_{\hspace{1pt} 2}^{\boldsymbol{\tau}}}\big)}\Big)_{\tau,\ep}$ is bounded in $\L^2(Q_T)$.
 Using assumption \textnormal{\textbf{H2}}, we see that $a_{12}$ is at most
linearly growing, so that using again the dual estimate \eqref{ineq:dual} (but inverting the subscripts $1$ and $2$), $\Big(a_{12}\big({\underline{u}_{\hspace{1pt} 2}^{\boldsymbol{\tau}}}\big)\Big)_{\tau,\ep}$ is bounded in $\L^2(Q_T)$. Writing
\begin{align*}
{\underline{u}_{\hspace{1pt} 1}^{\boldsymbol{\tau}}}a_{12}\big({\underline{u}_{\hspace{1pt} 2}^{\boldsymbol{\tau}}}\big) = {\underline{u}_{\hspace{1pt} 1}^{\boldsymbol{\tau}}}\sqrt{a_{12}\big({\underline{u}_{\hspace{1pt} 2}^{\boldsymbol{\tau}}}\big)} \sqrt{a_{12}\big({\underline{u}_{\hspace{1pt} 2}^{\boldsymbol{\tau}}}\big)},
\end{align*}
we see eventually that $\Big({\underline{u}_{\hspace{1pt} 1}^{\boldsymbol{\tau}}}a_{12}\big({\underline{u}_{\hspace{1pt} 2}^{\boldsymbol{\tau}}}\big)\Big)_{\tau,\ep}$ is bounded in $\L^{4/3}(Q_T)$ and we may thus extract a subsequence converging weakly in this space, and whose limit has to be equal to $u_1a_{12}(u_2)$ (because of the previous almost everywhere convergence). 
\vspace{2mm}\\
As for the self-diffusion, the dual estimate \eqref{ineq:dual} ensures that $\Big({\underline{u}_{\hspace{1pt} 1}^{\boldsymbol{\tau}}} \sqrt{d_{11}^\ep\big({\underline{u}_{\hspace{1pt} 1}^{\boldsymbol{\tau}}}\big)}\Big)_{\tau,\ep}$ is bounded in $\L^2(Q_T)$. Since $d_{11}^\ep$ is a nondecreasing function, we infer the boundedness of $\Big(a_{11}^\ep\big({\underline{u}_{\hspace{1pt} 1}^{\boldsymbol{\tau}}}\big)\Big)_{\tau,\ep}=\Big({\underline{u}_{\hspace{1pt} 1}^{\boldsymbol{\tau}}}d_{11}^\ep\big({\underline{u}_{\hspace{1pt} 1}^{\boldsymbol{\tau}}}\big)\Big)_{\tau,\ep}$ in $\L^1(Q_T)$, so that the weak convergence of this sequence can be deduced from its uniform integrability (Dunford-Pettis theorem). Since $d_{11}^\ep \leq d_{11}$ which is a nondecreasing continuous function, introducing the pseudo-inverse $g_{11}(t):=\inf\{x\in\R_+ ;x\,d_{11}(x) \geq t\}$ which is also going to $+\infty$ with $t$, we get
\begin{align*}
\int_0^T \int_\Omega {\underline{u}_{\hspace{1pt} 1}^{\boldsymbol{\tau}}} d_{11}^\ep\big({\underline{u}_{\hspace{1pt} 1}^{\boldsymbol{\tau}}}\big) \mathbf{1}_{ \underline{u}_{\hspace{1pt} 1}^{\boldsymbol{\tau}} d_{11}^\ep({\underline{u}_{\hspace{1pt} 1}^{\boldsymbol{\tau}}})\geq M}(t,x) \dd x \dd t &\leq \int_0^T \int_\Omega {\underline{u}_{\hspace{1pt} 1}^{\boldsymbol{\tau}}} d_{11}^\ep\big({\underline{u}_{\hspace{1pt} 1}^{\boldsymbol{\tau}}}\big) \mathbf{1}_{\underline{u}_{\hspace{1pt} 1}^{\boldsymbol{\tau}}d_{11}({\underline{u}_{\hspace{1pt} 1}^{\boldsymbol{\tau}}})\geq M}(t,x) \dd x \dd t \\
&= \int_0^T \int_\Omega {\underline{u}_{\hspace{1pt} 1}^{\boldsymbol{\tau}}}d_{11}^\ep\big({\underline{u}_{\hspace{1pt} 1}^{\boldsymbol{\tau}}}\big) \mathbf{1}_{{\underline{u}_{\hspace{1pt} 1}^{\boldsymbol{\tau}}}\geq g_{11}(M)}(t,x) \dd x \dd t\\
&\leq \frac{1}{g_{11}(M)} \left\|{\underline{u}_{\hspace{1pt} 1}^{\boldsymbol{\tau}}}\sqrt{d_{11}^\ep\big({\underline{u}_{\hspace{1pt} 1}^{\boldsymbol{\tau}}}\big)}\right\|_{\L^2(Q_T)}^2,
\end{align*}
which indeed goes to $0$ with $M$, uniformly in $\ep,\tau$, thanks to \eqref{ineq:dual}.
\vspace{2mm}\\
We hence get 
\begin{align*}
\int_0^T \int_\Omega \Big[ a_{11}^\ep\big({\underline{u}_{\hspace{1pt} 1}^{\boldsymbol{\tau}}}\big)+{\underline{u}_{\hspace{1pt} 1}^{\boldsymbol{\tau}}}a_{12}\big({\underline{u}_{\hspace{1pt} 2}^{\boldsymbol{\tau}}}\big)\Big]\cdot \Delta \theta (t,x) \, \dd x\,\dd t\operatorname*{\longrightarrow}_{(\tau,\ep)\rightarrow (0,0)} \int_0^T \int_\Omega \Big[a_{11}(u_1)+u_1a_{12}(u_2)\Big]\cdot \Delta \theta (t,x) \, \dd x\,\dd t.
\end{align*}
Since $\theta$ is smooth, we have
\begin{align*}
 \frac{\theta-\sigma_\tau \theta}{\tau} \operatorname*{\longrightarrow}_{\tau\rightarrow 0} -\partial_t \theta,
\end{align*}
uniformly on $Q_T$, and $\mathbf{1}_{[0,T-\tau[}$ converges to $1$ in all $\L^s(Q_T)$ ($s<\infty$), hence
\begin{align*}
 \int_0^{T}  \int_\Omega \mathbf{1}_{[0,T-\tau[}(t){\underline{u}_{\hspace{1pt} 1}^{\boldsymbol{\tau}}} \cdot \frac{\theta-\sigma_\tau \theta}{\tau} (t,x) \,\dd x\,\dd t \operatorname*{\longrightarrow}_{(\tau,\ep)\rightarrow (0,0)}  -\int_0^{T}  \int_\Omega u_1 \cdot \partial_t \theta (t,x) \,\dd x\,\dd t.
\end{align*}
For the reactions terms, since the nonlinearities are always strictly sublinear or dominated by the self diffusion,  one easily manages (using the dual estimate) to use Dunford Pettis criterion. Indeed, first, since $(\gamma_\ep)_\ep$ is increasing to the identity, 
the uniform integrability of $R_{12}^{-,\ep}(\underline{\uu}^{\boldsymbol{\tau}})$ reduces to check this property for both $s_{11}(\underline{u}_{\hspace{1pt} 1}^{\boldsymbol{\tau}})\underline{u}_{\hspace{1pt} 1}^{\boldsymbol{\tau}}$ and $s_{12}(\underline{u}_{\hspace{1pt} 2}^{\boldsymbol{\tau}})\underline{u}_{\hspace{1pt} 1}^{\boldsymbol{\tau}}$. Since $d_{11}$ is nondecreasing, we always have $\displaystyle \lim_{z \rightarrow +\infty}\frac{s_{11}(z)}{z d_{11}(z)}= 0$, so that  using again the pseudo-inverse $f_{11}(t):=\inf\{x\in\R_+ ;x\,s_{11}(x) \geq t\}$, which is also going to $+\infty$ with $t$, we get for $M>0$
\begin{align*}
\int_0^T \int_\Omega s_{11}(\underline{u}_{\hspace{1pt} 1}^{\boldsymbol{\tau}})\underline{u}_{\hspace{1pt} 1}^{\boldsymbol{\tau}} \mathbf{1}_{s_{11}(\underline{u}_{\hspace{1pt} 1}^{\boldsymbol{\tau}})\underline{u}_{\hspace{1pt} 1}^{\boldsymbol{\tau}} \geq M} (t,x) \dd x \dd t& = \int_0^T \int_\Omega s_{11}(\underline{u}_{\hspace{1pt} 1}^{\boldsymbol{\tau}})\underline{u}_{\hspace{1pt} 1}^{\boldsymbol{\tau}} \mathbf{1}_{\underline{u}_{\hspace{1pt} 1}^{\boldsymbol{\tau}} \geq f_{11}(M)} (t,x) \dd x \dd t \\
&\leq \sup_{z \geq f_{11}(M)} \frac{s_{11}(z)}{z d_{11}(z)} \left\|{\underline{u}_{\hspace{1pt} 1}^{\boldsymbol{\tau}}}\sqrt{d_{11}\big({\underline{u}_{\hspace{1pt} 1}^{\boldsymbol{\tau}}}\big)}\right\|_{\L^2(Q_T)}^2,
\end{align*}
which goes to $0$ with $M^{-1}$ uniformly in $\ep,\tau$ thanks to \eqref{ineq:dual}. We used here the fact that estimate \eqref{ineq:dual} is still true if one replaces $d_{11}^\ep$ by $d_{11}$ (Fatou's lemma). For $s_{12}(\underline{u}_{\hspace{1pt} 2}^{\boldsymbol{\tau}})\underline{u}_{\hspace{1pt} 1}^{\boldsymbol{\tau}}$, using
 again the pseudo-inverse,
 it is easy to exhibit some positive function $\ell$ going to $+\infty$ in $+\infty$ such as $s_{12}(\underline{u}_{\hspace{1pt} 2}^{\boldsymbol{\tau}})\underline{u}_{\hspace{1pt} 1}^{\boldsymbol{\tau}} \geq M$ large enough, $\underline{u}_{\hspace{1pt} 2}^{\boldsymbol{\tau}} \leq \ell(M) \Rightarrow \underline{u}_{\hspace{1pt} 1}^{\boldsymbol{\tau}} \geq \ell(M)$. In the case of the previous implication, one may easily use 
the already noticed uniform integrability of $(\underline{u}_{\hspace{1pt} 1}^{\boldsymbol{\tau}})_\tau$ and if $\underline{u}_{\hspace{1pt} 2}^{\boldsymbol{\tau}} \leq \ell(M)$,
 then we have by Young's inequality for any small $\delta >0$, 
$$ u_1\, s_{12}(u_2) \le \delta \, u_1^2\, a_{12}(u_2) + \frac{s_{12}^2(u_2)}{\delta\, a_{12}(u_2)}, $$
 so that we may conclude using \eqref{ineq:dual} and assumption \textbf{H1} :
\begin{align*}
\lim_{z \to +\infty} \frac{s_{12}(z)}{z\, \sqrt{d_{22}(z) + a_{12}(z)} } = 0.
\end{align*} 
\vspace{2mm}\\
For the other terms, first notice that 
\begin{align*}
\left| \ep \sum_{k=1}^N \int_{(k-1)\tau}^{k\tau}\langle w_1^k,\theta(t)-\Delta\theta(t)\rangle\,\dd t\right| &\leq \ep \|\theta-\Delta\theta\|_{\L^\infty([0,T]; \L^2(\Omega))} \sum_{k=1}^N \tau \|w_1^k\|_{\L^2(\Omega)}\\
&\leq  \|\theta-\Delta\theta\|_{\L^\infty([0,T]; \L^2(\Omega))} \ep\tau \sqrt{N}\left\{\sum_{k=1}^N \|w_1^k\|_{\L^2(\Omega)}^2\right\}^{1/2}\\
&\leq \|\theta-\Delta\theta\|_{\L^\infty([0,T]; \L^2(\Omega))} \ep\tau \sqrt{N} \frac{\sqrt{K_T(\mathscr{E}_\ep(\uu^0) +1)}}{\sqrt{\ep\tau}}\\
&= \|\theta-\Delta\theta\|_{\L^\infty([0,T]; \L^2(\Omega))} \sqrt{\ep} \sqrt{T} \sqrt{K_T(\mathscr{E}_\ep(\uu^0) +1)} \operatorname*{\longrightarrow}_{(\tau,\ep)\rightarrow (0,0)} 0,
\end{align*}
where we used \eqref{ineq:afterlim1} for the last but one inequality.\vspace{2mm}\\
We also have
\begin{align*}
\left|\int_0^T \int_\Omega \ep {\underline{u}_{\hspace{1pt} 1}^{\boldsymbol{\tau}}}{\underline{u}_{\hspace{1pt} 2}^{\boldsymbol{\tau}}} \Delta \theta (t,x) \, \dd x\,\dd t \right| \leq \ep \big\|{\underline{u}_{\hspace{1pt} 1}^{\boldsymbol{\tau}}}\big\|_{\L^2(Q_T)} \big\|{\underline{u}_{\hspace{1pt} 2}^{\boldsymbol{\tau}}}\big\|_{\L^2(Q_T)} \|\Delta \theta\|_{\L^\infty(Q_T)},
\end{align*}
which goes to zero with $(\tau,\ep)$ thanks to the dual estimate \eqref{ineq:dual}.\vspace{2mm}\\
Finally because of the continuity of $\theta$, we can write
\begin{align*}
 \frac{1}{\tau}\int_0^\tau \langle u^0_1, \theta(t)\rangle\,\dd t = \int_0^1  \langle u^0_1, \theta(\tau t)\rangle\,\dd t \operatorname*{\longrightarrow}_{(\tau,\ep)\rightarrow (0,0) } \langle u^0_1, \theta(\tau t)\rangle = \int_\Omega u^0_1(x)\theta(0,x)\,\dd x,
\end{align*}
so that it only remains to show 
\begin{align*}
\frac{1}{\tau}\int_0^\tau \langle u_1^1-u_1^0 , \theta(t)\rangle\,\dd t = \int_0^\tau \langle \partial_\tau u^1_1, \theta(t)\rangle\,\dd t \operatorname*{\longrightarrow}_{(\tau,\ep)\rightarrow (0,0) } 0.
\end{align*}
From equation \eqref{eq:weak_form} for $k=1$, $i=1$, $j=2$, we have ($\chi_1 = \theta(t)$)
\begin{align*}
\langle \partial_\tau u_1^1, \theta(t)\rangle- \left\langle \Big[ a_{11}^\ep(u^1_1)+u_1^1 a_{12}(u_2^1)+ \ep u_1^1 u_2^1 \Big],\Delta \theta(t) \right \rangle
 &+\ep \langle w_1^{1},\theta(t)- \Delta \theta(t)\rangle \\
&= \left\langle R_{12}^\ep(u_1^{1},u_2^{1}),\theta(t) \right\rangle,
\end{align*}
that gives, since $\L^1(\Omega)\hookrightarrow \W^{2,\infty}(\Omega)'$, for some positive constant $C_\theta$ depending on $\theta$,
\begin{align*}
\frac{1}{C_\theta}\left|\int_0^\tau \langle \partial_\tau u_1^1, \theta(t)\rangle \dd t\right| &\leq \tau \| a_{11}^\ep(u^1_1)+u_1^1 a_{12}(u_2^1)+ \ep u_1^1 u_2^1+\ep w_1^{1}+ R_{12}^\ep(u_1^{1},u_2^{1})\|_{\L^1(\Omega)}\\
&\leq  \left\|{\underline{u}_{\hspace{1pt} 1}^{\boldsymbol{\tau}}} {d_{11}^\ep\big({\underline{u}_{\hspace{1pt} 1}^{\boldsymbol{\tau}}}\big)}\mathbf{1}_{[0,\tau]\times\Omega}\right\|_{\L^1(Q_T)} + \|{\underline{u}_{\hspace{1pt} 1}^{\boldsymbol{\tau}}}a_{12}\big({\underline{u}_{\hspace{1pt} 2}^{\boldsymbol{\tau}}}\big)\mathbf{1}_{[0,\tau]\times\Omega}\|_{\L^1(Q_T)}\\
&+ \ep \|{\underline{u}_{\hspace{1pt} 1}^{\boldsymbol{\tau}}}\|_{\L^2(Q_T)}^2 + \ep \|{\underline{u}_{\hspace{1pt} 2}^{\boldsymbol{\tau}}}\|_{\L^2(Q_T)}^2 + \ep\tau \sqrt{\mu(\Omega)} \|w_1^1\|_{\L^2(\Omega)} +\|R^\ep_{12}(\underline{\uu}^{\boldsymbol{\tau}}) \mathbf{1}_{[0,\tau]\times\Omega}\|_{\L^1(Q_T)}.
\end{align*}
As shown before,
 all the sequences written in front of the characteristic
 function $\mathbf{1}_{[0,\tau]\times\Omega}$ are uniformly integrable 
(we actually showed this for the self diffusion term, and have some $\L^p(Q_T)$ with $p>1$ for the two others), and they hence vanish with $(\tau,\ep)$. The three other terms also vanish, 
because of the duality bound for the $\|\underline{u}_{\hspace{1pt i}}^{\boldsymbol{\tau}}\|_{\L^2(Q_T)}^2$ terms and because estimate \eqref{ineq:afterlim1} for the $\|w_1^1\|_{\L^2(\Omega)}$ term.
\vspace{2mm}\\
We get then 
\begin{align*}
\int_{Q_T} u_1 \Big\{\partial_t \theta - \big[d_{11}(u_1)+a_{22}(u_2)]\Delta\theta \Big\}\, \dd x\,\dd t =  \int_\Omega u^0_1(x)\theta(0,x)\,\dd x,
\end{align*}
that is the weak formulation on $Q_T$ of the equation
\begin{align*}
\partial_t u_1 -\Delta \Big[ a_{11}(u_1) + u_1 a_{12}(u_2)\Big] = R_{12}(u_1,u_2),
\end{align*}
initialized with $u_1(0,x) = u_1^0(x)$. A similar (symmetric) formulation holds for $u_2$.
\bigskip

We end up the proof of our Theorem with the passage to the limit in the duality estimates.
First notice that the constant $D$ in \eqref{ineq:dual} can be written as a
polynomial function (with positive coefficients) of $(\|\uu^0\|_{\L^1(\R)}, \mathscr{E}_\ep(\uu^0))$. Since $({\underline{u}_{\hspace{1pt} 1}^{\boldsymbol{\tau}}}\sqrt{d_{11}^\ep\big({\underline{u}_{\hspace{1pt} 1}^{\boldsymbol{\tau}}}\big)+a_{12}^\ep\big({\underline{u}_{\hspace{1pt} 2}^{\boldsymbol{\tau}}}\big)})_{(\ep,\tau)}$ converges almost everywhere to $u_1 \sqrt{d_{11}(u_1)+a_{12}(u_2)}$, and since $\mathscr{E}_\ep(\uu^0)$ converges to $\mathscr{E}(\uu^0)$, the classical weak estimate ensures that
\begin{align*}
\left\|u_1 \sqrt{d_{11}(u_1)+a_{12}(u_2)}\right\|_{\L^2(Q_T)}\leq \P(\|\uu^0\|_{\L^1(\Omega)},\mathscr{E}(\uu^0)), 
\end{align*}
for some polynomial function $\P$ with positive coefficients. 
\end{proof}

\bibliographystyle{abbrv}
\bibliography{DLM}

\end{document}